\documentclass[a4paper,12pt,leqno]{amsart}

\usepackage{amsmath}
\usepackage{amsmath}
\usepackage{amsthm}
\usepackage{amsfonts}
\usepackage{bm}
\usepackage{amssymb}
\usepackage{latexsym}
\usepackage{amsfonts}
\usepackage{graphicx}
\usepackage{mathrsfs}
\usepackage{amscd}
\usepackage{pifont}
\usepackage{enumerate}
\usepackage{color}
\usepackage{exscale}
\usepackage{paralist}
\usepackage{graphics}
\usepackage{epsfig}
\usepackage{graphicx}
\usepackage{epstopdf}
\usepackage{amsmath}
\usepackage{amssymb}
\usepackage{amsfonts}
\usepackage{graphicx}
\usepackage[colorlinks]{hyperref}
\renewcommand\eqref[1]{(\ref{#1})} 

\usepackage[margin=2.5cm]{geometry}

\newtheorem{theorem}{Theorem}[section]

\newtheorem{lemma}[theorem]{Lemma}

\theoremstyle{definition}
\newtheorem{definition}[theorem]{Definition}

\newtheorem{remark}[theorem]{Remark}

\numberwithin{equation}{section}

\newcommand{\supp}{\operatorname{supp}}
\newcommand{\diam}{\operatorname{diam}}
\newcommand{\Dom}{\operatorname{Dom}}

\begin{document}

\title[Besov and Triebel-Lizorkin spaces associated to operators]{Continuous
characterizations of inhomogeneous Besov and Triebel-Lizorkin spaces
 associated to non-negative self-adjoint operators}

\author[Q. Hong]{Qing Hong}
\address{Qing Hong
\endgraf
School of Mathematics and Statistics, Jiangxi Normal University
\endgraf
Nanchang, Jiangxi 330022, China}
\email{qhong@mail.bnu.edu.cn}

\author[G. Hu]{Guorong Hu$^\ast$}
\address{Guorong Hu
\endgraf
School of Mathematics and Statistics, Jiangxi Normal University
\endgraf
 Nanchang, Jiangxi 330022, China}
\email{hugr@mail.ustc.edu.cn}

\subjclass[2010]{Primary 46E35; Secondary 42B25, 42B35}

\date{\today}

\keywords{Besov spaces, Triebel-Lizorkin spaces, metric measure spaces, heat kernel, Littlewood-Paley function, Lusin function}

\thanks{This project is supported by the National Natural Science Foundation of China  (Grant Nos. 11901256 and
120012511) and the Natural Science Foundation of Jiangxi Province  (Grant Nos. 20192BAB211001 and 20202BAB211001)}
\thanks{$^\ast$Corresponding author. Email:hugr@mail.ustc.edu.cn}

\begin{abstract}
Let $(M, \rho, \mu)$ be a metric measure space satisfying the doubling, reverse doubling and noncollapsing conditions,
and let $\mathscr{L}$ be a nonnegative self-adjoint operator acting on $L^2 (M, d\mu)$ whose heat kernel satisfies
the small-time Gaussian upper bound, H\"{o}lder continuity and Markov property. In this paper,
we establish new characterizations of the ``classical'' and ``nonclassical'' Besov and Triebel-Lizorkin spaces
associated to $\mathscr{L}$ introduced by Kerkyacharian and Petrushev. More precisely, we obtain characterizations
of these spaces in terms of continuous Littlewood-Paley and Lusin functions associated to the heat semigroup generated
by $\mathscr{L}$, for complete range of indices. This extends related known results in the classical Euclidean setting
to our general setting, and extends corresponding results in [Trans. Amer. Math. Soc. \textbf{367} (2015), 121--189]
to complete range of indices.
\end{abstract}

\maketitle

\section{Introduction and statement of main results} \label{sec1}
\allowdisplaybreaks

In the last two decades the study of function spaces associated to operators attracted significant attention.
This direction of study was initiated by Auscher et al. \cite{ADM}, who introduced the Hardy space
$H^{1}_{L}(\mathbb{R}^{n})$ associated to an operator $L$ with pointwise heat kernel bound.
Later Duong and Yan \cite{DY1, DY2} introduced BMO space associated to
operators and investigated the duality between $H^{1}_{L}(\mathbb{R}^{n})$ and ${BMO}_{L^{\ast}}
(\mathbb{R}^{n})$, where $L^{\ast}$ is the adjoint of $L$ in $L^{2}(\mathbb{R}^{n})$. For Hardy spaces
associated to operators without pointwise heat kernel bound, we refer to the works of Auscher et al. \cite{AMR}
and Hofmann and Mayboroda \cite{HM}, in which Hardy spaces associated to the Hodge
Laplacian on Riemannian manifolds, and Hardy spaces associated to second order divergence form elliptic operators on $\mathbb{R}^{n}$
with complex coefficients, were developed respectively. Motivated by these two works,
Hofmann {et al.} \cite{HLMMY} further established the theory of the Hardy spaces $H^{p}_{{L}}(X)$, $1 \leq p <\infty$, on a
metric measure space $(X, d, \mu)$ associated to a general nonnegative self-adjoint operator ${L}$
satisfying Davies-Gaffney estimates. For further developments concerning Hardy spaces associated to operators,
we refer to \cite{BD, DKKP, DL, DLY, JY1, MP, SY, Yan, YY},  among many others.

It is well known that Besov and Triebel-Lizorkin spaces form a unifying class of function spaces encompassing many well-studied
classical function spaces such as Lebesgue spaces $L^{p}$, Hardy spaces $H^{p}$, the space BMO,
Sobolev spaces, Hardy-Sobolev spaces and various forms of Lipschitz spaces. While the classical theory of
these spaces on $\mathbb{R}^{n}$ was developed primarily by Peetre, Triebel, Frazier, Jawerth and many other authors
(see for instance \cite{FJ1, FJ2, FJ3, P, T1, T2}), there have been many efforts of extending it to
other domains and nonclassical settings.
In particular, Kerkyacharian and Petrushev \cite{KP} recently developed Besov and Triebel-Lizorkin spaces associated
to nonnegative self-adjoint operators. To recall the definition of these spaces, let us fix the setting.
Assume that $(M,\rho, \mu)$ is a metric measure space, locally compact with respect to
the topology induced by the distance $\rho(\cdot, \cdot)$, satisfying the following conditions.

\textbf{(H1)} \textit{Doubling condition}: There exists a constant $c_{0} >1$ such that
\begin{align} \label{D}
0 < \mu(B(x,2r))\leq c_{0} \mu(B(x,r)) < \infty
\end{align}
for all $x \in M$ and $r >0$, where $B(x,r):= \{y \in M: \rho(x,y) <r\}$ is the open ball centered at $x$ of radius $r$.

\textbf{(H2)} \textit{Reverse doubling condition}: There exists a constant $c_1 > 1$ such that
\begin{align*}
\mu(B(x,2r)) \geq c_{1}\mu(B(x,r))
\end{align*}
for all $x \in M$ and $0< r \leq \frac{\diam M}{3}$.

\textbf{(H3)} \textit{Non-collapsing condition}:  There exists a constant $c_2 > 0$ such that
\begin{align*}
\inf_{x\in M}\mu(B(x, 1)) \geq c_{2}.
\end{align*}

Assume further that the geometry of $(M, \rho, \mu)$ is related to a nonnegative self-adjoint operator $\mathscr{L}$
acting on $L^{2}(M,d\mu)$, mapping real-valued functions to
real-valued functions, such that the associated semigroup $P_{t} = e^{-t\mathscr{L}}$ consists of
integral operators with (heat) kernel $p(t, x, y)$ obeying the following conditions.

\textbf{(H4)} {\it Small time Gaussian upper bound}: There exist constants $C, c>0$ such that
\begin{align} \label{56}
|p(t,x,y)| &\leq \frac{C}{\sqrt{\mu\big(B(x,\sqrt{t})\big)\mu\big(B(y,\sqrt{t})\big)}}
\exp \left\{-\frac{c\rho^{2}(x,y)}{t} \right\}
\end{align}
for $x, y \in M$ and $0 < t \leq 1$.

\textbf{(H5)} \textit{H\"{o}lder continuity}: There exist constants $C, c, \alpha >0$ such that
\begin{equation} \label{57}
\begin{split}
&|p(t,x,y)-p(t, x,y')| \\
&\quad \leq C\left( \frac{\rho(y,y')}{\sqrt{t}} \right)^{\alpha}
\frac{1}{\sqrt{\mu\big(B(x,\sqrt{t})\big)\mu\big(B(y,\sqrt{t})\big)}}\exp \left\{-\frac{c\rho^{2}(x,y)}{t} \right\}
\end{split}
\end{equation}
for $x,y,y' \in M$ and $0 <t \leq 1$, whenever $\rho(y,y') \leq \sqrt{t}$.

\textbf{(H6)} \textit{Markov property}:
\begin{align*}
\int_{M}p(t, x,y)d\mu(y) = 1
\end{align*}
for all $x \in M$ and $t >0$.

In what follows, we always assume that $(M, \rho, \mu)$ is a metric measure space, locally compact with respect to
the topology induced by the distance $\rho(\cdot, \cdot)$, satisfying \textbf{(H1), (H2), (H3)},
and $\mathscr{L}$ is a nonnegative self-adjoint operator on $L^{2}(M, d\mu)$ with heat kernel satisfying
\textbf{(H4), (H5), (H6)}.  This setting is quite general and covers a wide range of situations,
including Lie groups of polynomial volume growth and complete Riemannian manifolds with Ricci curvature bounded
from below and satisfying the volume doubling
condition. In particular, it covers the classical Euclidean setting. See \cite{CKP} and
\cite{KP} for more details.

We now recall the definition of Besov and Triebel-Lizorkin spaces associated to $\mathscr{L}$, introduced by
Kerkyacharian and Petrushev \cite{KP}. Denote by $\mathcal{S}_{\mathscr{L}}'$ the class of distributions associated to
$\mathscr{L}$, which was introduced in \cite{KP} and  will be recalled in Section \ref{sec2} below.
\begin{definition} \label{df1}
Let $\varphi_{0}, \varphi \in C^{\infty}_{0}(\mathbb{R})$ be even functions such that
\begin{align} \label{908}
\supp \varphi_{0} \subset \{\lambda \in \mathbb{R}: |\lambda| \leq 2\} ~\mbox{ and }~  |\varphi_{0}(\lambda)|
\geq c >0  \mbox{ for } |\lambda| \leq 2^{3/4}
\end{align}
and
\begin{align} \label{909}
\supp \varphi \subset \{\lambda \in \mathbb{R}: 1/2 \leq |\lambda| \leq 2\} ~\mbox{ and }~  |\varphi(\lambda)|
\geq c >0 \mbox{ for } 2^{-3/4} \leq |\lambda| \leq 2^{3/4}.
\end{align}

(i) Let $s \in \mathbb{R}$, $0<p\leq \infty$ and $0 <q \leq \infty$. The ``classical'' Besov
space $B^{s}_{p,q}(\mathscr{L})$
is defined as the collection of all $f \in \mathcal{S}_{\mathscr{L}}'$ such that
\begin{align*}
\|f\|_{B^{s}_{p,q}(\mathscr{L})}:= \|\varphi_{0}(\sqrt{\mathscr{L}})f\|_{L^{p}}+
\left( \sum_{j =1}^{\infty} 2^{jsq} \big\|\varphi(2^{-j}\sqrt{\mathscr{L}})f\big\|_{L^{p}}^{q}   \right)^{1/q} <\infty,
\end{align*}
with the usual modification when $q = \infty$. See Section \ref{sec2} for the definition of $\varphi_{0}(\sqrt{\mathscr{L}})f$
and $\varphi(2^{-j}\sqrt{\mathscr{L}})f$.
The ``nonclassical'' Besov space $\widetilde{B}^{s}_{p,q}(\mathscr{L})$
is defined as the collection of all $f \in \mathcal{S}_{\mathscr{L}}'$ such that
\begin{align*}
\|f\|_{\widetilde{B}^{s}_{p,q}(\mathscr{L})} &:=\big\||B(\cdot, 1)|^{-s/n}\varphi_{0}(\sqrt{\mathscr{L}})
f(\cdot)\big\|_{L^{p}} \\
& \quad\quad  + \left( \sum_{j =1}^{\infty} \big\| |B(\cdot, 2^{-j})|^{-s/n}\varphi
(2^{-j}\sqrt{\mathscr{L}})f(\cdot)\big\|_{L^{p}}^{q}
  \right)^{1/q} <\infty.
\end{align*}
Here, $n$ is the ``dimension'' of $(M, \rho, \mu)$ as in \eqref{D2} below.

(ii) Let $s \in \mathbb{R}$, $0 <p < \infty$ and $0<q \leq \infty$. The ``classical'' Triebel-Lizorkin
space $F^{s}_{p,q}(\mathscr{L})$ is defined as the collection of all $f \in \mathcal{S}_{\mathscr{L}}'$ such that
\begin{align*}
\|f\|_{F^{s}_{p,q}(\mathscr{L})} :=\|\varphi_{0}(\sqrt{\mathscr{L}})f\|_{L^{p}}+
 \left\|\left(\sum_{j =1}^{\infty} 2^{jsq} \big|\varphi(2^{-j}\sqrt{\mathscr{L}})f\big|^{q}\right)^{1/q}\right\|_{L^{p}}
< \infty.
\end{align*}
The ``nonclassical'' Triebel-Lizorkin space $\widetilde{F}^{s}_{p,q}(\mathscr{L})$ is defined as the
collection of all $f \in \mathcal{S}_{\mathscr{L}}'$ such that
\begin{align*}
\|f\|_{\widetilde{F}^{s}_{p,q}(\mathscr{L})} &:=\big\||B(\cdot, 1)|^{-s/n}\varphi_{0}(\sqrt{\mathscr{L}})
f(\cdot)\big\|_{L^{p}} \\
&\quad\quad  + \left\|\left(\sum_{j =1}^{\infty} |B(\cdot, 2^{-j})|^{-sq /n}
 \big|\varphi(2^{-j}\sqrt{\mathscr{L}})f(\cdot)\big|^{q}\right)^{1/q}\right\|_{L^{p}} < \infty.
\end{align*}
\end{definition}

As pointed out in \cite{KP}, the main motivation for introducing the ``nonclassical''
Besov and Triebel-Lizorkin spaces associated to $\mathscr{L}$ lies in nonlinear approximation.
These ``nonclassical'' spaces seem more suitable for the possibly anisotropic nature of the geometry of $(M, \rho,
\mu)$.

Homogeneous Besov and Triebel-Lizorkin spaces associated to nonnegative self-adjoint spaces have been
introduced and studied by Georgiadis et al. \cite{GKKP, GKKP2}, and their weighted extension were studied by
Bui et al. \cite{BBD}. It is worth noting that in \cite{BBD} neither the H\"{o}lder continuity nor
the Markov property for the heat kernel of the operator is assumed.

The main purpose of the present paper is to derive characterizations of the inhomogeneous spaces $B^{s}_{p,q}(\mathscr{L})$, $\widetilde{B}^{s}_{p,q}(\mathscr{L})$,
 $F^{s}_{p,q}(\mathscr{L})$ and $\widetilde{F}^{s}_{p,q}(\mathscr{L})$ in terms of continuous Littlewood-Paley and Lusin
functions associated to the heat semigroup generated by $\mathscr{L}$.
Characterization of the classical Besov and Triebel-Lizorkin spaces on $\mathbb{R}^{n}$ via the continuous
Littlewood-Paley function associated to the heat semigroup were studied by Flett \cite{F}, Peetre \cite{P},
Triebel \cite{T3}, Bui {et al.} \cite{Bui2, Bui3, Bui4}, and many other authors. Kerkyacharian and Petrushev \cite{KP} proved
such a characterization for $B^{s}_{p,q}(\mathscr{L})$, $\widetilde{B}^{s}_{p,q}(\mathscr{L})$,
 $F^{s}_{p,q}(\mathscr{L})$ and $\widetilde{F}^{s}_{p,q}(\mathscr{L})$, but with the restriction
$p \geq 1$ in the Besov case, and with the restriction $p>1$ and $q>1$ in the Triebel-Lizorkin
case (see \cite[Theorems 6.7 and 7.5]{KP}). In \cite{LYY}, Liu et al. obtained the heat semigroup
characterization of more general scales of functions/distributions  associated to $\mathscr{L}$,
called Besov-type and Triebel-Lizorkin-type spaces associated to $\mathscr{L}$, for complete range of indices.
However, the heat semigroup characterization for complete range of indices in \cite{LYY} is a discrete version,
rather than a continuous one (see \cite[Theorem 5.7]{LYY}). The continuous version of heat semigroup
characterization obtained there still needs the restriction $p \geq 1$ (see \cite[Theorem 5.8]{LYY}).
Only recently, Bui et al. \cite{BBD} proved continuous characterizations of
(weighted) Besov and Triebel-Lizorkin spaces associated to $\mathscr{L}$ for complete range of indices in
terms of Littlewood-Paley and Lusin functions.
Compared with \cite{BBD}, the novelty of our work lies in the following aspects. First,
we treat both ``classical'' and ``nonclassical'' spaces, while \cite{BBD} only treated ``classical'' ones.
Obviously the ``nonclassical'' spaces are harder to handle than the ``classical'' ones. Second, in our work we focus on the inhomogeneous spaces, while
\cite{BBD} only treated the homogeneous spaces. The continuous characterizations of inhomogeneous spaces are more difficult to prove.
This is because in many parts of the argument, the inhomogeneous term need to be treated separately.

The main results of the present paper are the following two theorems.

\begin{theorem} \label{heat}
Suppose $s \in \mathbb{R}$ and $0<q \leq \infty$. Let $m$ be the smallest integer
such that $m > \max\{s/2, 0\}$. Then the following statements are true.

{\rm (i)} If $0< p \leq \infty$, then for $f \in \mathcal{S}_{\mathscr{L}}'$, we have the quasi-norm equivalence:
\begin{align} \label{hqy00}
\|f\|_{B^{s}_{p,q}(\mathscr{L})} &\sim  \|e^{-\mathscr{L}}f\|_{L^{p}} +
\left( \int_{0}^{1} t^{-sq/2}\big\|(t\mathscr{L})^{m}e^{-t\mathscr{L}}f\big\|^{q}_{L^{p}}  \frac{dt}{t}\right)^{1/q}
\end{align}
and
\begin{equation} \label{hqy1}
\begin{split}
 \|f\|_{\widetilde{B}^{s}_{p,q}(\mathscr{L})} &\sim \big\||B(\cdot, 1)|^{-s/n}e^{-\mathscr{L}}f(\cdot)\big\|_{L^{p}} \\
&\quad + \left(\int_{0}^{1} \big\| |B(\cdot, t^{1/2})|^{-s/n}(t\mathscr{L})^{m}e^{-t\mathscr{L}}f(\cdot)\big\|_{L^{p}}^{q}
 \frac{dt}{t}\right)^{1/q}.
\end{split}
\end{equation}

{\rm (ii)}
If $0< p < \infty$, then for $f \in \mathcal{S}_{\mathscr{L}}'$, we have the quasi-norm equivalence:
\begin{align} \label{hqy000}
\|f\|_{F^{s}_{p,q}(\mathscr{L})} &\sim \|e^{-\mathscr{L}}f\|_{L^{p}} +
\left\|\left(\int_{0}^{1} t^{-sq/2} \big|(t\mathscr{L})^{m}e^{-t\mathscr{L}}f\big|^{q} \frac{dt}{t}\right)^{1/q} \right\|_{L^{p}}
\end{align}
and
\begin{equation} \label{hqy2}
\begin{split}
 \|f\|_{\widetilde{F}^{s}_{p,q}(\mathscr{L})} &\sim \big\||B(\cdot, 1)|^{-s/n}e^{-\mathscr{L}}f(\cdot)\big\|_{L^{p}} \\
 &\quad + \left\|\left(\int_{0}^{1}|B(\cdot, t^{1/2})|^{-sq/n}
 \big|(t\mathscr{L})^{m}e^{-t\mathscr{L}}f(\cdot)\big|^{q} \frac{dt}{t}\right)^{1/q} \right\|_{L^{p}}.
\end{split}
\end{equation}
\end{theorem}

\begin{theorem} \label{area}
Suppose $s \in \mathbb{R}$, $0< p < \infty$ and $0<q \leq \infty$.  Let $m$
be the smallest integer such that $m > \max\{s/2, 0\}$.
Then for $f \in \mathcal{S}_{\mathscr{L}}'$, we have the quasi-norm equivalence:
\begin{equation} \label{lusin1}
\begin{split}
\|f\|_{F^{s}_{p,q}(\mathscr{L})} &\sim \|e^{-\mathscr{L}}f\|_{L^{p}} \\
&\quad + \left\|\left(\iint_{\Gamma^{loc}(\cdot)}t^{-sq/2} \big|(t\mathscr{L})^{m}e^{-t\mathscr{L}}f(y)
\big|^{q}\frac{d\mu(y)}{|B(\cdot, t^{1/2})|}\frac{dt}{t}\right)^{1/q} \right\|_{L^{p}}
\end{split}
\end{equation}
and
{\small \begin{equation} \label{lusin2}
\begin{split}
 \|f\|_{\widetilde{F}^{s}_{p,q}(\mathscr{L})} &\sim \big\||B(\cdot, 1)|^{-s/n}e^{-\mathscr{L}}f(\cdot)\big\|_{L^{p}} \\
&\quad + \left\|\left(\iint_{\Gamma^{loc}(\cdot)}|B(y, t^{1/2})|^{-sq/n}
 \big|(t\mathscr{L})^{m}e^{-t\mathscr{L}}f(y)\big|^{q} \frac{d\mu(y)}{|B(\cdot, t^{1/2})|}\frac{dt}{t}\right)^{1/q} \right\|_{L^{p}},
\end{split}
\end{equation}
where} $\Gamma^{loc}(x):=\{(y,t) \in M \times (0,1] :  \rho(x,y) <t^{1/2}\}$.
\end{theorem}
A few words about our proofs are in order.  We point out that the approach in \cite{KP} to derive continuous characterizations
in terms of square functions can not be used to treat the case $p <1$.
To achieve our goal, we shall adapt the ideas developed by Bui {et al.} \cite{Bui2,Bui3,Bui4},
Rychkov \cite{Rychkov} and Ullrich \cite{U}.
A central estimate of our approach to derive
continuous characterizations of Besov and Triebel-Lizorkin spaces associated to operators
is a ``sub-mean value inequality'' involving $t$ (see Lemma \ref{moui} below), which generalizes the
corresponding inequality on $\mathbb{R}^{n}$.
We also need to make good use of the smooth functional calculus developed by Kerkyacharian and Petrushev in \cite{KP}
and the ``off-diagonal estimates'' proved in Section \ref{sec3}.
It is worth mentioning that
the continuous characterizations of inhomogeneous spaces do not follow directly from the approach for
the homogeneous spaces in \cite{BBD}. Indeed, in many situations it is not a trivial matter to handle the inhomogeneous term.
For example, to prove the sub-mean value inequality for the inhomogeneous term,
a separate nontrivial argument is needed  (see the proof of Lemma \ref{moui}). Moreover, note the in the continuous norms of inhomogeneous spaces, the inhomogeneous term
is independent of $t$. Thus, to prove that the discrete norms can be bounded by the continuous norms, we
need to construct a Calder\'{o}n reproducing identity in which the inhomogeneous term is independent of $t$ (see \eqref{nontrivial} in the proof of Lemma~\ref{Lemma2}).
Finally, we emphasize that our proof of continuous Lusin characterization of ``nonclassical'' inhomogeneous Triebel-Lizorkin spaces (Theorem \ref{area}) is significantly
different from the method used in \cite{BBD}.

It is worth pointing out that Besov and Triebel-Lizorkin spaces associated to some particular operators were earlier
studied by some authors. For instance, Besov and Triebel-Lizorkin spaces in the context of Hermite were studied by
Petrushev and Xu \cite{PX} and Bui and Duong \cite{BD0}, while these spaces in the context of Laguerre were studied
by Kerkyacharian {et al.} \cite{KPX} and Bui and Duong \cite{BD}. The spaces in \cite{BD0} and \cite{BD} were introduced
via continuous Littlewood-Paley functions associated to the heat semigroup (or Poisson semigroup) and, in some
restricted cases (e.g., $q =2$), their Lusin function characterization was obtained.

\medskip
The organization of this paper is as follows. In Section \ref{sec2}, we give some notions and preliminary
results which will be needed in the proofs of our main results. In Section \ref{sec3} we present off-diagonal estimates,
which could be regarded as refinements of the previously known ones. Section \ref{sec4} and Section \ref{sec5}
are devoted to the proofs of Theorem \ref{heat} and Theorem \ref{area}, respectively.

\medskip
\textit{Notation.} Throughout this article we shall use the notation $|E| := \mu(E)$ for any measurable set $E \subset M$.
The set of all nonnegative integers is denoted by $\mathbb{N}_{0}$, while the set of all strictly positive integers
is denoted by $\mathbb{N}$. For any positive number $\alpha$, we denote by $\lfloor \alpha \rfloor$
the largest integer less than or equal to $\alpha$. We shall also use the notation $L^{p}: =L^{p}(M, d\mu)$.
In some cases ``sup'' will mean ``ess sup'', which will be clear from the context.
We will use $c,C,c',C'$ to denote positive constants, which are independent of the main
variables involved and whose values may vary at every occurrence.
By writing $f \lesssim g$ or $g \gtrsim f$, we mean $f \leq Cg$.
The notation $f \sim g$ will stand for $C \leq f/g \leq C'$.

\section{Preliminaries and notation} \label{sec2}

We start by noting that the doubling condition \eqref{D} implies the following strong homogeneity property:
there exists $C>0$ such that
\begin{equation} \label{D2}
|B(x,\lambda r)| \leq C \lambda^{n}|B(x,r)|
\end{equation}
for all $x \in M$, $r >0$ and $\lambda \geq 1$,
where $n$ is a constant playing the role of a dimension, though it is not even an integer.
There also exit $C$ and $n'$, $0 \leq n' \leq n$, so that
\begin{equation} \label{iuy}
|B(x,r)| \leq C \left(1+\frac{\rho(x,y)}{r} \right)^{n'}|B(y,r)|
\end{equation}
uniformly for all $x, y \in M$ and $r >0$. Indeed, property \eqref{iuy} with $n'=n$
is a direct consequence of the triangle inequality for the metric $\rho$ and the strong homogeneity property \eqref{D2}.
In the case of the Euclidean space $\mathbb{R}^{n}$ and Lie groups of polynomial growth, $n'$ can be chosen to be $0$.

Using the doubling condition \eqref{D}, it is easy to show (cf. \cite[Lemma 2.3]{CKP})
that for any $\sigma >n$, there exists a constant
$C$ such that for all $x\in M$ and $t>0$,
\begin{equation}\label{inte}
\int_{M} \left(1+\frac{\rho(x,y)}{t}\right)^{-\sigma} d\mu(y) \leq C|B(x,t)|.
\end{equation}

To save space we shall use the following abbreviation borrowed from \cite{KP}:
\begin{equation*}
D_{t, \sigma} (x,y): = \big(|B(x,t)||B(y,t)| \big)^{-1/2} \left( 1+\frac{\rho(x,y)}{t} \right)^{-\sigma}
\end{equation*}
for $t, \sigma >0$ and $x, y \in M$.
Combining \eqref{iuy} and \eqref{inte} we see that for any $\sigma > n+ n'/2$, there is a constant $C$
(depending on $\sigma$) such that
\begin{equation} \label{ges}
\int_{M}D_{t,\sigma}(x,y)d\mu(y) \leq C
\end{equation}
uniformly for all $t \in (0, \infty)$ and $x \in M$.

The following lemma is standard and thus we skip the proof.
\begin{lemma} \label{maxxx}
Suppose $\sigma >n + n'$. Then there exists a constant $C>0$ such that for all locally integrable functions $f$ on $M$,
$t >0$ and $x \in M$,
\begin{align*}
\int_{M}\frac{|f(y)|}{|B(y,t)|(1 + t^{-1}\rho(x,y))^{\sigma}}d\mu(y) \leq C \mathcal{M}(f)(x).
\end{align*}
Here $\mathcal{M}$ is the Hardy-Littlewood maximal operator on $(M,\rho, \mu)$ defined by
\begin{equation*}
\mathcal{M}(f)(x):= \sup_{B \ni x}\frac{1}{\mu(B)} \int_{B}|f(y)|d\mu(y),
\end{equation*}
where $B$ ranges over all balls containing $x$.
\end{lemma}

It is well known that the Fefferman-Stein vector-valued maximal inequality also holds on metric measure spaces satisfying
doubling condition (see, e.g., \cite{GLY}). It is stated as follows.
\begin{lemma} \label{fsv}
Suppose that $1 <p <\infty$ and $1 < q \leq \infty$. Then there
exists a constant $C>0$ such that for all sequences $\{f_{j}\}$ of locally integrable functions on $M$,
\begin{equation*}
\left\|\left(\sum_{j}| \mathcal{M}(f_{j})|^{q}\right)^{1/q} \right\|_{L^{p}} \leq C \left\|\left(\sum_{j}
| f_{j}|^{q}\right)^{1/q} \right\|_{L^{p}} .
\end{equation*}
\end{lemma}

The Besov and Triebel-Lizorkin spaces introduced in \cite{KP} are in general spaces of distributions.
So let us recall from \cite{KP} the notions of test functions and distributions on $M$ associated to $\mathscr{L}$.
\begin{definition}
\textrm{(i)}  If $\mu(M)< \infty$, the test function space
$\mathcal{S}_{\mathscr{L}}$ is defined as the collection of all functions $f \in \bigcap_{k \in \mathbb{N}_{0}}\Dom(\mathscr{L}^{k})$ with
topology induced by the family of seminorms
\begin{align*}
\mathcal{P}_{k}(f) := \|\mathscr{L}^{k}f\|_{L^{2}}, \quad k \in \mathbb{N}_{0}.
\end{align*}

\textrm{(ii)} If $\mu(M) = \infty$,
the test function space $\mathcal{S}_{\mathscr{L}}$ is defined as the
collection of all functions $f \in \bigcap_{k \in \mathbb{N}_{0}}\Dom(\mathscr{L}^{k})$ such that
\begin{equation} \label{semin}
\mathcal{P}_{k,\ell}(f) := \sup_{x \in M} (1+\rho(x,x_{0}))^{\ell} |\mathscr{L}^{k}f(x)|< \infty \quad
\mbox{for all } k, \ell \in \mathbb{N}_{0},
\end{equation}
where $x_{0} \in M$ is a fixed point. In this case $\mathcal{S}_{\mathscr{L}}$ is endowed with
the topology induced by the family $\{\mathcal{P}_{k,\ell}\}_{k, \ell \in \mathbb{N}_{0}}$ of seminorms.
\end{definition}

In either case, $\mathcal{S}_{\mathscr{L}}$ is a Fr\'{e}chet space (see \cite[Section 5]{KP}). Moreover,
in the case where $\mu(M) = \infty$, a different choice of $x_{0}$ in the above definition yields the same
space $\mathcal{S}_{\mathscr{L}}$ with equivalent topology. Thus, we fix the point $x_{0} \in M$ once and for
all.

The space $\mathcal{S}_{\mathscr{L}}'$ of distributions associated to $\mathscr{L}$ is
defined as the space of all continuous linear functionals on $\mathcal{S}_{\mathscr{L}}$.
The duality between the spaces is denoted by the map $$(\cdot,\cdot): \mathcal{S}_{\mathscr{L}}' \times \mathcal{S}_{\mathscr{L}}
\rightarrow \mathbb{C}.$$

Given a bounded Borel measurable function $\phi$ on $[0, \infty)$, one can define the
operator $\phi(\sqrt{\mathscr{L}})$ by the spectral theorem, according to the prescription
\begin{equation*}
\phi(\sqrt{\mathscr{L}}) = \int_{0}^{\infty} \phi(\sqrt{\lambda})dE_{\lambda},
\end{equation*}
where $dE_{\lambda}$ is the projection valued measure associated to $\mathscr{L}$.
Note that if $\varphi, \psi$ are two bounded measurable functions on $[0, \infty)$,
and $\phi:= \varphi \psi$, then
\begin{align*}
\varphi (\sqrt{\mathscr{L}}) \psi(\sqrt{\mathscr{L}}) =  \psi(\sqrt{\mathscr{L}})  \varphi(\sqrt{\mathscr{L}}) = \phi (\sqrt{\mathscr{L}}).
\end{align*}
Moreover, if $\varphi(\sqrt{\mathscr{L}})$ and $\psi(\sqrt{\mathscr{L}})$ are integral operators with kernels $K_{\varphi(\sqrt{\mathscr{L}})}(x,y)$
and $K_{\psi(\sqrt{\mathscr{L}})}(x,y)$ respectively, then $\varphi(\sqrt{\mathscr{L}}) \psi(\sqrt{\mathscr{L}})$ is also an integral operator whose
kernel is given by
\begin{align} \label{compokernel}
K_{\varphi(\sqrt{\mathscr{L}}) \psi(\sqrt{\mathscr{L}})} (x,y) = \int_M K_{\varphi(\sqrt{\mathscr{L}})}(x,z) K_{\psi(\sqrt{\mathscr{L}})}(z,y)d\mu(z).
\end{align}

Given $N \in \mathbb{N}$ and $\phi \in C^{N}(\mathbb{R})$, we introduce the seminorm
\begin{equation} \label{swce}
\|\phi\|_{(N)} :=\sup_{\lambda \in \mathbb{R}, 0<\nu \leq N} (1 + |\lambda|)^{N+n+1}\big|\phi^{(\nu)}(\lambda)\big|,
\end{equation}
where $\phi^{(\nu)}$ is the $\nu$-th order derivative of $\phi$.

The following result concerning smooth functional calculus plays an important role in our approach.
This was developed by Kerkyacharian and Petrushev \cite{KP}.
\begin{lemma} \label{AOE1} {\rm (\cite[Theorem 3.4]{KP})}
Let $N \in \mathbb{N}$ and $N \geq n+1$. Suppose $\phi \in C^{N}(\mathbb{R})$ is an even function such that $\|\phi\|_{(N)} < \infty$.
Then for any $t >0$, $\phi(t \sqrt{\mathscr{L}})$ is an integral operator, and its integral
kernel $K_{\phi (t\sqrt{L})}(x,y)$ satisfies
\begin{align*}
\big|K_{\phi(t\sqrt{\mathscr{L}})}(x,y)\big| \leq C\|\phi\|_{(N)}D_{t, N}(x,y),
\end{align*}
where $C>0$ is a constant depending on $N$ and the constants $c_{0}, C^{\star}, c^{\star}$
from \eqref{D}--\eqref{57}.
\end{lemma}

\begin{remark}
\textrm{This lemma implies that if $\phi$ is an even Schwartz function on $\mathbb{R}$ and $t >0$,
then both $K_{\phi(t\sqrt{\mathscr{L}})}(\cdot, y)$ and $K_{\phi(t\sqrt{\mathscr{L}})}(x, \cdot)$ belong
to the test function class $\mathcal{S}_{\mathscr{L}}$. Thus, for $f \in \mathcal{S}_{\mathscr{L}}'$,
it is natural to define
\begin{equation*}
\phi(\sqrt{\mathscr{L}})f(x) : = \big(f, K_{\phi(\sqrt{\mathscr{L}})}(x,\cdot) \big), \quad x \in M,
\end{equation*}
which plays a role similar to the convolution of a test function and a distribution in the Euclidean space $\mathbb{R}^n$.}
\end{remark}

Given $a>0$, $\gamma \in \mathbb{R}$, $t >0$, $f \in \mathcal{S}_{\mathscr{L}}'$ and an even function
$\phi \in \mathcal{S}(\mathbb{R})$, we introduce the Peetre type maximal functions:
\begin{align*}
\big[\phi(t\sqrt{\mathscr{L}})\big]^{\ast}_{a}f(x)&:= \sup_{y \in M} \frac{|\phi(t\sqrt{\mathscr{L}})f(y)|}{(1+t^{-1}\rho(x,y))^{a}},\\
\big[\phi(t\sqrt{\mathscr{L}})\big]^{\ast}_{a, \gamma}f(x)&:= \sup_{y \in M} \frac{|B(y,t)|^{\gamma}|\phi(t\sqrt{\mathscr{L}})f(y)|}{(1+t^{-1}\rho(x,y))^{a}}.
\end{align*}
Observe that $\big[\phi(t\sqrt{\mathscr{L}})\big]^{\ast}_{a}f(x)=\big[\phi(t\sqrt{\mathscr{L}})\big]^{\ast}_{a,0}f(x)$.

The following lemma follows from the Peetre type inequality (cf. \cite[Lemma 6.4]{KP}), Hardy-Littlewood maximal inequality
and Fefferman-Stein vector-valued maximal inequality (Lemma \ref{fsv} above).
See also the proofs of Proposition~6.3 and Proposition~7.2 in \cite{KP}.
\begin{lemma} \label{ptre}
Let $\varphi_{0}, \varphi \in C^{\infty}_{0}(\mathbb{R})$ be even functions satisfying \eqref{908} and \eqref{909}.

\begin{itemize}
\item[\rm (i)]  For $s\in \mathbb{R}$, $0<p\leq \infty$, $0 < q \leq \infty$, $a > \frac{n}{p}$ and $f \in \mathcal{S}_{\mathscr{L}}'$, we have
\begin{align*}
&\quad \left\|\big[\varphi_{0}(\sqrt{\mathscr{L}}) \big]^{\ast}_{a}f\right\|_{L^{p}} + \left(\sum_{j=1}^{\infty}2^{jsq}\left\| \big[\varphi(2^{-j}\sqrt{\mathscr{L}}) \big]^{\ast}_{a}f\right\|_{L^{p}}^{q} \right)^{1/q}
\leq C \|f\|_{B^{s}_{p,q}(\mathscr{L})},\\
&\left\|\big[\varphi_{0}(\sqrt{\mathscr{L}}) \big]^{\ast}_{a, -s/n}f\right\|_{L^{p}}+
\left( \sum_{j=1}^{\infty}\left\| \big[\varphi(2^{-j}\sqrt{\mathscr{L}}) \big]^{\ast}_{a, -s/n}f\right\|_{L^{p}}^{q} \right)^{1/q}
\leq C \|f\|_{\widetilde{B}^{s}_{p,q}(\mathscr{L})}.
\end{align*}

\item[\rm (ii)] For $s\in \mathbb{R}$, $0<p<\infty$, $0 < q \leq \infty$, $a > \frac{n}{\min\{p,q\}}$
and $f \in \mathcal{S}_{\mathscr{L}}'$, we have
\begin{align*}
& \quad \left\|\big[\varphi_{0}(\sqrt{\mathscr{L}}) \big]^{\ast}_{a}f\right\|_{L^{p}}
+ \left\|\left( \sum_{j=1}^{\infty}2^{jsq}\left| \big[\varphi(2^{-j}\sqrt{\mathscr{L}}) \big]^{\ast}_{a}f\right|^{q}
\right)^{1/q}\right\|_{L^{p}} \leq C \|f\|_{F^{s}_{p,q}(\mathscr{L})}, \\
&\left\|\big[\varphi_{0}(\sqrt{\mathscr{L}}) \big]^{\ast}_{a, -s/n}f\right\|_{L^{p}}+\left\|\left( \sum_{j=1}^{\infty}
\left| \big[\varphi(2^{-j}\sqrt{\mathscr{L}}) \big]^{\ast}_{a, -s/n}f\right|^{q} \right)^{1/q}\right\|_{L^{p}}
\leq C \|f\|_{\widetilde{F}^{s}_{p,q}(\mathscr{L})}.
\end{align*}
\end{itemize}
\end{lemma}

We will need the following fundamental lemma from \cite{Rychkov}.
\begin{lemma} \label{sfga}
{\rm (\cite[Lemma 2]{Rychkov})} Let $0<p,q\leq \infty$ and $\delta >0$.
Let $\{g_{j}\}_{j = 0}^{\infty}$ be a sequence of nonnegative
measurable function on $M$ and put
\begin{equation*}
G_{\ell}(x) =\sum_{j=0}^{\infty}2^{-|j-\ell|\delta}g_{j}(x), \quad x \in M,
 ~\ell\in \mathbb{N}_{0}.
\end{equation*}
Then, there is a constant $C$ depending only on $p,q,\delta$ such that
\begin{align*}
\big\|\{ G_{\ell}\}_{\ell=0}^{\infty}\big\|_{\ell^{q}(L^{p})}
&\leq C\big\| \{ g_{j}\}_{j=0}^{\infty}\big\|_{\ell^{q}(L^{p})}
\end{align*}
and
\begin{align*}
\big\|\{ G_{\ell}\}_{\ell=0}^{\infty}\big\|_{L^{p}(\ell^{q})}
&\leq C\big\| \{ g_{j}\}_{j=0}^{\infty}\big\|_{L^{p}(\ell^{q})}. 
\end{align*}
Here, the $\ell^{q}(L^{p})$ quasi-norm and $L^{p}(\ell^{q})$ quasi-norm are, respectively, given by
\begin{align*}
\big\|\{ h_{j}\}_{j=0}^{\infty}\big\|_{\ell^{q}(L^{p})}:= \left( \sum_{j=0}^{\infty} \|h_{j}\|_{L^{p}}^{q}
\right)^{1/q}\quad \mbox{and} \quad
\big\|\{ h_{j}\}_{j=0}^{\infty}\big\|_{L^{p}(\ell^{q})}:=\left\|\left(\sum_{j=0}^{\infty}|h_{j}|^{q}
\right)^{1/q} \right\|_{L^{p}},
\end{align*}
for any sequence $\{h_{j}\}_{j=0}^{\infty}$
of measurable functions on $M$.
\end{lemma}

The following simple lemma will also be needed.
\begin{lemma} \label{osigg}
Let $0<p,q\leq \infty$ and $\delta >0$. Let $\{g_{j}\}_{j = 0}^{\infty}$ be a sequence of nonnegative
measurable function on $M$ and put
\begin{equation*}
f(x) =\sum_{j=0}^{\infty}2^{-j\delta}g_{j}(x), \quad x \in M.
\end{equation*}
Then, there is a constant $C$ depending only on $p,q,\delta$ such that
\begin{align*}
\|f\|_{L^{p}} \leq C \left\| \left\| \big\{g_{j}(\cdot) \big\}_{j=0}^{\infty}\right\|_{\ell^{q}}  \right\|_{L^{p}}
\end{align*}
\end{lemma}

\begin{proof}
First assume $q >1$. By H\"{o}lder's inequality we have
\begin{align*}
|f(x)| \leq \left(\sum_{j=0}^{\infty}2^{-j\delta q'}\right)^{1/q'} \left(\sum_{j=0}^{\infty}
|g_{j}(x)|^{q}\right)^{1/q} \leq C \left\|\big\{g_{j}(\cdot) \big\}_{j=0}^{\infty}\right\|_{\ell^{q}} .
\end{align*}
Taking the $L^{p}$-quasi-norm on both sides yields the desired estimate.

If $0<q \leq 1$, we use the inequality $(\sum_j u_j)^q \leq \sum_j |u_j|^q$ to obtain
\begin{align*}
|f(x)| \leq \sum_{j=0}^{\infty}2^{-j\delta}|g_{j}(x)|
\leq \left(\sum_{j=0}^{\infty} 2^{-j \delta q}  |g_j (x)|^q   \right)^{1/q}
\leq \left(\sum_{j=0}^{\infty} |g_j (x)|^q   \right)^{1/q} .
\end{align*}
Taking the $L^{p}$-quasi-norm on both sides yields the desired estimate.
\end{proof}

Finally, we record a Calder\'{o}n type reproducing formula.
\begin{lemma} {\rm (\cite[Proposition 5.5]{KP})} \label{Calderon}
Suppose $\phi_{0}, \phi \in C^{\infty}_{0}(\mathbb{R})$ are even functions such that
$\supp \phi_{0} \in \{\lambda \in \mathbb{R}:|\lambda| \leq 2\}$, $\supp \phi \subset \{ \lambda \in \mathbb{R}:
1/2 \leq |\lambda| \leq 2\}$ and
\begin{align*}
\phi_{0} (\lambda) + \sum_{j =1}^{\infty} \phi(2^{-j} \lambda) =1
\end{align*}
for all $\lambda \in \mathbb{R}$. Then for all $f \in \mathcal{S}_\mathscr{L}'$, we have
\begin{align*}
f = \phi_{0}(\sqrt{\mathscr{L}})f +  \sum_{j =1}^{\infty} \phi(2^{-j}\sqrt{\mathscr{L}})f
\end{align*}
with convergence in the topology of $\mathcal{S}_{\mathscr{L}}'$.
\end{lemma}

\section{Off-diagonal estimates} \label{sec3}
First, we establish the following fundamental estimate:
\begin{lemma} \label{owr}
For any $\sigma > n + n'$, there exists a constant $c>0$ such that for all
$t,s>0$ and all $x,y \in M$,
\begin{equation} \label{imsy}
\int_{M}D_{t,\sigma}(x,z)D_{s,\sigma}(z,y)d\mu(z) \leq c D_{t \vee s, \sigma-n-n'}(x,y),
\end{equation}
where $t \vee s : =\max\{t,s\}$.
\end{lemma}
\begin{proof}
By symmetry, we only need to show \eqref{imsy} for $t \geq  s$. To do this,
we decompose
\begin{align*}
\int_{M}D_{t,\sigma}(x,z)D_{s,\sigma}(z,y)d\mu(z) = \left(\int_{\Omega_{1}}
+ \int_{\Omega_{2}}\right) D_{t,\sigma}(x,z)D_{s,\sigma}(z,y)d\mu(z)=:I_{1} + I_{2},
\end{align*}
where $\Omega_{1}:= \{z \in M: \rho(y, z)< \rho(x,y)/2\}$ and $\Omega_{2}:=\{z \in M:
\rho(y,z) \geq \rho(x,y)/2\}$.
By the triangle inequality for the distance $\rho(\cdot, \cdot)$ we have $\rho(x,z) \geq \rho(x, y) / 2$ for all $z \in \Omega_{1}$.
From this and \eqref{iuy} we see that for all $z \in \Omega_{1}$,
\begin{equation} \label{mot}
\begin{split}
D_{t,\sigma}(x,z)
& = \big(|B(x,t)||B(z,t)|\big)^{-1/2}(1+t^{-1}\rho(x,z))^{-\sigma} \\
& \lesssim \big(|B(x,t)||B(y,t)|\big)^{-1/2}(1+t^{-1}\rho(y,z))^{n'/2}(1+t^{-1}\rho(x,z))^{-\sigma}  \\
& \lesssim \big(|B(x,t)||B(y,t)|\big)^{-1/2}(1+t^{-1}\rho(x,y))^{-\sigma+ n'/2}\\
& =D_{t, \sigma- n'/2}(x,y).
\end{split}
\end{equation}
This along with \eqref{ges} yields that
\begin{equation} \label{wtu1}
\begin{split}
I_{1} \lesssim D_{t, \sigma-n'/2}(x,y) \int_{\Omega_{1}} D_{s, \sigma} (z,y)d\mu(z)
\lesssim D_{t, \sigma-n'/2}(x,y) \leq D_{t, \sigma-n-n'}(x,y).
\end{split}
\end{equation}

Next we estimate $I_{2}$. Note that
by \eqref{iuy} and the elementary inequality
\begin{equation*}
1+t^{-1}\rho(y,z) \lesssim (1+t^{-1}\rho(x,y))(1+t^{-1}\rho(x,z)),
\end{equation*}
we have, for all $ z \in \Omega_{2}$,
\begin{equation} \label{ngil}
\begin{split}
D_{t,\sigma}(x,z)
& =  \big(|B(x,t)||B(z,t)|\big)^{-1/2}(1+t^{-1}\rho(x,z))^{-\sigma} \\
& \lesssim  \big(|B(x,t)||B(y,t)|\big)^{-1/2}(1+t^{-1}\rho(y,z))^{n'/2}(1+t^{-1}\rho(x,z))^{-n'/2} \\
& \lesssim  \big(|B(x,t)||B(y,t)|\big)^{-1/2}(1+t^{-1}\rho(x,y))^{n'/2}.
\end{split}
\end{equation}
Hence
\begin{equation} \label{jfl}
I_{2} \lesssim \big(|B(x,t)||B(y,t)|\big)^{-1/2}(1+t^{-1}\rho(x,y))^{n'/2} \int_{\Omega_{2}}
D_{s,\sigma} (z,y)d\mu(z).
\end{equation}
To proceed we consider two cases: $\rho (x,y) \leq t$ and $\rho(x,y) >t$.

If $\rho(x,y) \leq t$, then $1+t^{-1}\rho(x,y) \sim 1$, and hence it follows from \eqref{jfl} and \eqref{ges}  that
\begin{align} \label{ixw}
I_{2} \lesssim  \big(|B(x,t)||B(y,t)|\big)^{-1/2} \sim D_{t,\sigma-n-n'}(x,y).
\end{align}

If $\rho(x,y) > t$, we decompose the set $\Omega_{2}$ into
$\Omega_{2} = \bigcup_{k =0}^{\infty} E_{k}$, where
\begin{equation*}
E_{k}:= \{z \in M: 2^{k-1}\rho(x,y) \leq \rho(z,y) <2^{k}\rho(x,y)\}.
\end{equation*}
Then by \eqref{D2}, \eqref{iuy} and the fact that $t \geq s$,  we have
\begin{align*}
\int_{\Omega_{2}}D_{s, \sigma}&(z,y)d\mu(z) \lesssim |B(y,s)|^{-1}
\int_{\Omega_{2}} (1+s^{-1}\rho(z,y))^{-\sigma + n'/2}d\mu(z) \\
& \lesssim |B(y,s)|^{-1}s^{\sigma-n'/2} \sum_{k=0}^{\infty}\int_{E_{k}}
\rho(z,y)^{-\sigma + n'/2}d\mu(z) \\
& \leq s^{\sigma- n'/2}\sum_{k=0}^{\infty}
[2^{k-1}\rho(x,y)]^{-\sigma + n'/2} \frac{\big|B\big(y,2^{k}\rho(x,y)\big)\big|}{|B(y,s)|}\\
& \leq s^{\sigma-n'/2}\sum_{k=0}^{\infty}
[2^{k-1}\rho(x,y)]^{-\sigma + n'/2} \left(\frac{2^{k}\rho(x,y)}{s} \right)^{n}\\
& \lesssim t^{\sigma-n-n'/2}\rho(x,y)^{-\sigma + n+n'/2}\sum_{k=0}^{\infty}2^{-k(\sigma -n -n'/2)}   \\
& \lesssim (t^{-1}\rho(x,y))^{-\sigma+n + n'/2} \sim (1+t^{-1}\rho(x,y))^{-\sigma+n + n'/2}.
\end{align*}
 Inserting this estimate into \eqref{jfl} we obtain
\begin{equation} \label{wtu2}
I_{2} \lesssim D_{t, \sigma-n-n'}(x,y).
\end{equation}

Therefore, in either case we have $I_{2}\lesssim D_{t, \sigma-n-n'}(x,y)$, which together with
\eqref{wtu1} yields \eqref{imsy}. The proof is complete.
\end{proof}

\begin{lemma} \label{AOE}
Suppose $\phi, \psi$ are even Schwartz functions on $\mathbb{R}$ such that
$(\cdot)^{-2k}\phi(\cdot)  \in \mathcal{S}(\mathbb{R})$ for some positive integer $k$.
Then, if $s \leq t$, we have
 \begin{align} \label{2mg}
 \big|K_{\phi(s\sqrt{\mathscr{L}})\psi(t\sqrt{\mathscr{L}})}(x,y) \big| \leq C\big\|(\cdot)^{-2k}\phi(\cdot) \big\|_{(N)}
 \big\|(\cdot)^{2k}\psi(\cdot)\big\|_{(N)}\left( \frac{s}{t}\right)^{2k}D_{t, N-n-n'}(x,y),
 \end{align}
where $\|\cdot\|_{(N)}$ is defined by \eqref{swce}.
\end{lemma}
\begin{proof}
Since
\begin{equation*}
\phi(s\sqrt{\mathscr{L}})\psi(t\sqrt{\mathscr{L}})=\left(\frac{s}{t} \right)^{2k}
\big[(s\sqrt{\mathscr{L}})^{-2k}\phi(s\sqrt{\mathscr{L}})\big] \big[(t\sqrt{\mathscr{L}})^{2k} \psi(t \sqrt{\mathscr{L}})\big],
\end{equation*}
by \eqref{compokernel}, Lemma \ref{AOE1} and Lemma \ref{owr}, we have
\begin{align*}
&\big|K_{ \phi(s\sqrt{\mathscr{L}})\psi(t\sqrt{\mathscr{L}})}(x,y)\big|\\
& =\left(\frac{s}{t} \right)^{2k}
 \left|\int_{M} K_{ (s\sqrt{\mathscr{L}})^{-2k}\phi(s\sqrt{\mathscr{L}})}(x,z)
K_{(t\sqrt{\mathscr{L}})^{2k} \psi(t \sqrt{\mathscr{L}}) } (z,y)d\mu(z) \right|\\
& \leq C\big\|(\cdot)^{-2k}\phi(\cdot) \big\|_{(N)}
 \big\|(\cdot)^{2k}\psi(\cdot)\big\|_{(N)}\left(\frac{s}{t} \right)^{2k}
\int_{M}  D_{s,N}(x,z)D_{t,N}(z,y)d\mu(z) \\
& \leq C\big\|(\cdot)^{-2k}\phi(\cdot) \big\|_{(N)}
 \big\|(\cdot)^{2k}\psi(\cdot)\big\|_{(N)}\left(\frac{s}{t} \right)^{2k}D_{t, N-n-n'}(x,y),
\end{align*}
as desired.
\end{proof}
\begin{remark}
{\rm Compared with \cite[Lemma 2.1 (ii)]{LYY}, the advantage of Lemma \ref{owr} is that
it avoids the appearance of the factor $\max\{(s/t)^{n}, (t/s)^{n}\}$ on the right-hand side of \eqref{imsy}.
Consequently, the factor $(t/s)^{n}$ does not appear on the right-hand side of \eqref{2mg}. (Compare with \cite[Proposition 2.14]{LYY}.)
This simple but important refinement yields the sufficiency of the condition $m > \max\{s/2, 0\}$ in Theorems \ref{heat} and \ref{area}.}
\end{remark}

\begin{remark}
{\rm Obviously, if we assume in addition that $(\cdot)^{-2k}\psi(\cdot)  \in \mathcal{S}(\mathbb{R})$
in Lemma \ref{AOE}, then for all $s, t >0$, we have
\begin{align*}
& \big|K_{\phi(s\sqrt{\mathscr{L}})\psi(t\sqrt{\mathscr{L}})}(x,y) \big|\\
&\hspace{1.5cm} \leq C  \max\left\{\big\|(\cdot)
 ^{-2k}\phi(\cdot) \big\|_{(N)} \big\|(\cdot)^{2k}\psi(\cdot) \big\|_{(N)}, \   \big\|(\cdot)^{2k}\phi(\cdot)\big\|_{(N)}\big\|(\cdot)^{-2k}\psi(\cdot) \big\|_{(N)}\right\} \\
& \hspace{2cm}\times  \left( \frac{s}{t} \wedge \frac{t}{s}\right)^{2k}D_{s \vee t, N-n-n'}(x,y).
\end{align*}
As a consequence, for all $j, \ell \in \mathbb{Z}$,
\begin{align*}
& \big|K_{\phi(2^{-j}\sqrt{\mathscr{L}})\psi(2^{-\ell}\sqrt{\mathscr{L}})}(x,y) \big|\\
&\hspace{1.5cm} \leq C  \max\left\{\big\|(\cdot)^{-2k}\phi(\cdot)
\big\|_{(N)}  \big\|(\cdot)^{2k}\psi(\cdot) \big\|_{(N)}, \   \big\|(\cdot)^{2k}\phi(\cdot) \big\|_{(N)} \big\|(\cdot)^{-2k}\psi(\cdot)
\big\|_{(N)}\right\} \\
&\hspace{2cm}  \times 2^{-2k |j -\ell|}D_{2^{-j \wedge \ell}, N-n-n'}(x,y).
 \end{align*}}
\end{remark}

\section{Proof of Theorem \ref{heat}} \label{sec4}


Let $s \in \mathbb{R}$ and let $m$ be the smallest integer such that $m > \max\{s/2, 0\}$.
Throughout this section, $\omega_0$ and $\omega$ are two fixed even Schwartz functions on $\mathbb{R}$ given by
\begin{align} \label{uytr}
\omega_{0}(\lambda): =e^{-\lambda^{2}} \quad \mbox{and} \quad \omega(\lambda):= \lambda^{2m} e^{-\lambda^{2}}, \quad \lambda \in \mathbb{R}.
\end{align}
Thus we have $\omega_0 (\sqrt{\mathscr{L}}) = e^{-\mathscr{L}}$ and $\omega(t^{1/2} \sqrt{\mathscr{L}}) = (t \mathscr{L})^m e^{-t \mathscr{L}}$.

The central estimates of our approach is in the following lemma.
\begin{lemma} \label{moui}
Let $r > 0$ and  $N\in \mathbb{N}$. Then there is a positive constant $C=C(r, N)$ such that for all $f \in \mathcal{S}_{\mathscr{L}}'$,
$x \in M$, $t \in [1,4]$ and $\ell \geq 1$,
\begin{equation} \label{cen222}
\big|\omega(2^{-\ell}t^{1/2}\sqrt{\mathscr{L}})f (x)\big|^{r}
\leq C\sum_{j=0}^{\infty} 2^{-2Nrj}  \int_{M}  \frac{|\omega(2^{-(j+\ell)}t^{1/2}\sqrt{\mathscr{L}})f(z)|^{r}}
{{|B(z,2^{-(j+\ell)})|}(1+2^{\ell}\rho(x,z))^{Nr}} d\mu(z)
\end{equation}
and
\begin{equation} \label{cen333}
\begin{split}
\big| \omega_{0}(\sqrt{\mathscr{L}})f(x)\big|^{r}
&\leq C\left( \int_{M} \frac{|\omega_{0}(\sqrt{\mathscr{L}})f(z)|^{r}}{|B(z,1)|(1+\rho(x,z))^{Nr}} d\mu(z) \right.\\
&\quad\quad\quad\quad\quad \left. + \sum_{j=1}^{\infty} 2^{-2Nrj } \int_{M}
\frac{ |\omega(2^{-j}t^{1/2}\sqrt{\mathscr{L}})f(z)|^{r}}{|B(z,2^{-j})|(1+\rho(x,z))^{Nr}} d\mu(z)\right).
\end{split}
\end{equation}
\end{lemma}
\begin{proof}
The ideas of such estimates in $\mathbb{R}^n$ were originated in \cite{Bui2,Bui3,Bui4}.
We will follow these ideas. Our proof is also inspired by \cite{Rychkov,U}.

Choose nonnegative even functions $\eta_{0}, \eta \in \mathcal{S}(\mathbb{R})$ such that
\begin{align*}
\eta_{0}(\lambda) \neq 0  \Longleftrightarrow |\lambda| < 2
\quad \ \mbox{and} \quad \  \eta(\lambda) \neq 0 \Longleftrightarrow 1/2 < |\lambda| < 2.
\end{align*}
Then  set
\begin{align*}
\zeta(\lambda):= \eta_{0}(\lambda)\omega_{0}(\lambda)
+\sum_{\ell =1}^{\infty}\eta(2^{-\ell}\lambda)\omega(2^{-\ell}\lambda), \quad
\lambda \in \mathbb{R}.
\end{align*}
Note that $\zeta(\lambda) >0$ for every $\lambda \in \mathbb{R}$. Put
$\psi_{0}(\lambda):= \eta_{0}(\lambda)/\zeta(\lambda)$ and
$\psi(\lambda):= \eta(\lambda)/\zeta(\lambda)$.
Then $\psi_{0}, \psi$ are even Schwartz functions on $\mathbb{R}$ satisfying that
$\supp \psi_{0} \subset \{\lambda \in \mathbb{R}: |\lambda| \leq 2\}$, $\supp \psi \subset
\{\lambda \in \mathbb{R}: 1/2 \leq |\lambda| \leq 2\}$, and
\begin{equation} \label{hpw00}
\omega_{0}({\lambda})\psi_{0}({\lambda})+ \sum_{j =1}^{\infty}\omega(2^{-j}{\lambda})
\psi(2^{-j}{\lambda})=1, \quad \forall \lambda \in \mathbb{R}.
\end{equation}
Setting $\omega_j (\lambda) := \omega(2^{-j}\lambda)$ and $\psi_j (\lambda):= \psi(2^{-j}\lambda)$ for $j \geq 1$,
we rewrite \eqref{hpw00} as
\begin{equation} \label{hpw}
\sum_{j =0}^{\infty}\omega_j(\lambda)
\psi_j (\lambda)=1, \quad \forall \lambda \in \mathbb{R}.
\end{equation}
Replacing $\lambda$ with $2^{-\ell}t^{1/2} \lambda$ in \eqref{hpw}, we get that for all
$\ell \in \mathbb{N}_0$ and $t\in [1,4]$,
\begin{equation*}
\sum_{j =0}^{\infty}\omega_j(2^{-\ell}t^{1/2}\lambda)
\psi_j (2^{-\ell}t^{1/2}\lambda)=1, \quad \forall \lambda \in \mathbb{R}.
\end{equation*}
It then follows from Lemma \ref{Calderon} that for all $f \in \mathcal{S}_{\mathscr{L}}'$,
\begin{equation*}
f = \sum_{j =0}^{\infty}\omega_j(2^{-\ell}t^{1/2}\sqrt{\mathscr{L}})
\psi_j (2^{-\ell}t^{1/2}\sqrt{\mathscr{L}})f
\end{equation*}
with convergence in the topology of $\mathcal{S}_{\mathscr{L}}'$.
Hence, for every $\ell \in \mathbb{N}_0$, we have the pointwise representation
\begin{align} \label{pointwiserep}
\omega_\ell (t^{1/2}\sqrt{\mathscr{L}})f (y)
= \sum_{j =0}^{\infty}\omega_\ell (t^{1/2}\sqrt{\mathscr{L}}) \omega_j(2^{-\ell}t^{1/2}\sqrt{\mathscr{L}})
\psi_j (2^{-\ell}t^{1/2}\sqrt{\mathscr{L}})f(y), \quad y \in M.
\end{align}
For $j, \ell \in \mathbb{N}_0$, we define
\begin{align*}
\theta_{j, \ell} (\lambda):= \begin{cases}
\omega_0 (2^{-\ell} \lambda), & j =0, \  \ell \in \mathbb{N}_0,  \\
\omega_\ell (\lambda), & j \in \mathbb{N}, \  \ell \in \mathbb{N}_0.
\end{cases}
\end{align*}
One can check that
\begin{align*}
\omega_\ell (t^{1/2} \lambda) \omega_j (2^{-\ell} t^{1/2}\lambda)= \theta_{j, \ell} (t^{1/2}\lambda) \omega_{j + \ell} (t^{1/2} \lambda),
\quad \forall j,\ell \in \mathbb{N}_0, \ t \in [1,4].
\end{align*}
Hence we can rewrite \eqref{pointwiserep} as
\begin{equation} \label{9900}
\begin{split}
\omega_\ell (t^{1/2}\sqrt{\mathscr{L}})f (y)
&= \sum_{j =0}^{\infty}
\psi_j (2^{-\ell}t^{1/2}\sqrt{\mathscr{L}})\theta_{j, \ell} (t^{1/2}\sqrt{\mathscr{L}}) \omega_{j + \ell} (t^{1/2} \sqrt{\mathscr{L}}) f (y) \\
& = \sum_{j =0}^{\infty} \int_{M}
K_{\psi_j (2^{-\ell}t^{1/2}\sqrt{\mathscr{L}})\theta_{j, \ell} (t^{1/2}\sqrt{\mathscr{L}})}(y,z) \omega_{j + \ell} (t^{1/2} \sqrt{\mathscr{L}}) f (z)d\mu(z).
\end{split}
\end{equation}

Let $N \in \mathbb{N}$ with $N > n +3n'/2$.
Note that $(\cdot)^{2N}\omega(t^{1/2}\cdot) \in \mathcal{S}(\mathbb{R})$,  $(\cdot)^{-2N}
\psi(t^{1/2} \cdot) \in \mathcal{S}(\mathbb{R})$ and there is a constant $c_{N}$ such that
\begin{align*}
 \sup_{t \in [1,4]} \big\|(\cdot)^{2N}\omega(t^{1/2} \cdot)\big\|_{(N)} \leq c_{N} \quad \mbox{and} \quad
\sup_{t \in [1,4]} \big\|(\cdot)^{-2N}\psi(t^{1/2} \cdot)\big\|_{(N)} \leq c_{N}.
\end{align*}
Hence, by Lemma \ref{AOE}, we have
\begin{align} \label{1902}
\left|K_{\psi_j (2^{-\ell}t^{1/2}\sqrt{\mathscr{L}})\theta_{j, \ell} (t^{1/2}\sqrt{\mathscr{L}})}(y,z) \right| \leq c_{N}' 2^{-2Nj} D_{2^{-\ell},N-n-n'}(y,z),
 \quad \forall j, \ell \in \mathbb{N}_0.
\end{align}
Inserting \eqref{1902} into \eqref{9900}, and using \eqref{iuy}, we obtain that for $\ell \in \mathbb{N}_0$, $t \in [1,4]$ and $y \in M$,
\begin{align*}
|\omega_\ell (t^{1/2}\sqrt{\mathscr{L}})f (y)|
\leq C_{N} \sum_{j=0}^{\infty} 2^{-2Nj}\int_{M}  \frac{|\omega_{j + \ell}(t^{1/2}\sqrt{\mathscr{L}})f(z)|}
{{|B(z,2^{-\ell})|}(1+2^{\ell}\rho(y,z))^{N-n-3n'/2}} d\mu(z) .
\end{align*}
Obviously, this implies that for any $N \in \mathbb{N}$,
\begin{align} \label{xog}
|\omega_\ell (t^{1/2}\sqrt{\mathscr{L}})f (y)|
 \leq \widetilde{C}_{N} \sum_{j=0}^{\infty} 2^{-2Nj}\int_{M}  \frac{|\omega_{j + \ell}(t^{1/2}\sqrt{\mathscr{L}})f(z)|}
{{|B(z,2^{-\ell})|}(1+2^{\ell}\rho(y,z))^{N}} d\mu(z),
\end{align}
where $\widetilde{C}_{N}:= C_{\lfloor N + n+3n'/2\rfloor +1}$.
Replacing $\ell$ by $i + \ell$ ($i \in \mathbb{N}_{0}$), and multiplying on both sides of \eqref{xog} by $2^{-2Ni}$, we get
\begin{equation} \label{nhe}
\begin{split}
&2^{-2Ni} |\omega_{i + \ell}(t^{1/2}\sqrt{\mathscr{L}})f (y)|\\
& \leq \widetilde{C}_{N}\sum_{j=0}^{\infty} 2^{-2N(j+i)} \int_{M}  \frac{|\omega_{j + i + \ell}(t^{1/2}\sqrt{\mathscr{L}})f(z)|}
{{|B(z,2^{-(i+\ell)})|}(1+2^{i+\ell}\rho(y,z))^{N}} d\mu(z)\\
& = \widetilde{C}_{N}\sum_{j=i}^{\infty} 2^{-2Nj} \int_{M}  \frac{|\omega_{j+\ell}(t^{1/2}\sqrt{\mathscr{L}})f(z)|}
{{|B(z,2^{-(i+\ell)})|}(1+2^{i+\ell}\rho(y,z))^{N}} d\mu(z)\\
&\leq \widetilde{C}_{N}\sum_{j=0}^{\infty} 2^{-2Nj} \int_{M}  \frac{|\omega_{j+\ell}(t^{1/2}\sqrt{\mathscr{L}})f(z)|}
{{|B(z,2^{-(j+\ell)})|}(1+2^{\ell}\rho(y,z))^{N}} d\mu(z).
\end{split}
\end{equation}
We divide both sides of \eqref{nhe} by $(1 + 2^{\ell}\rho(x,y))^{N}$, and use the inequality
\begin{align*}
(1 + 2^{\ell}\rho(x,y))^{N}(1+2^{\ell}\rho(y,z))^{N} \geq (1 +2^{\ell}\rho(x,z))^{N},
\end{align*}
to obtain
{\small \begin{equation} \label{pyw}
2^{-2Ni} \frac{|\omega_{i+\ell}(t^{1/2}\sqrt{\mathscr{L}})f (y)|}{(1 + 2^{\ell}\rho(x,y))^{N}  }
\leq \widetilde{C}_{N}\sum_{j=0}^{\infty} 2^{-2Nj}  \int_{M}  \frac{|\omega_{j+\ell}(t^{1/2}\sqrt{\mathscr{L}})f(z)|}
{{|B(z,2^{-(j+\ell)})|}(1+2^{\ell}\rho(x,z))^{N}} d\mu(z).
\end{equation}}

To prove the desired estimates we first consider the case $0<r \leq 1$. Define
{\small \begin{align*}
M_{\ell,t, N}f(x):= \sup_{i \geq 0} \sup_{y \in M} 2^{-2Ni} \frac{\big|\omega_{i+\ell}(t^{1/2}
\sqrt{\mathscr{L}})f(y)\big|}{(1 +2^{\ell}\rho(x,y))^{N}}, \quad \ell \in \mathbb{N}_0, \  t \in [1,4], \ N >0, \ x \in M.
\end{align*}
Then} \eqref{pyw} implies that
{\small \begin{equation} \label{wlrt}
M_{\ell, t, N}f(x) \leq \widetilde{C}_{N} \big[M_{\ell,t, N}f(x)\big]^{1-r}\sum_{j=0}^{\infty} 2^{-2Nrj}
\int_{M}  \frac{|\omega_{j+\ell}(t^{1/2}\sqrt{\mathscr{L}})f(z)|^{r}}
{{|B(z,2^{-(j+\ell)})|}(1+2^{\ell}\rho(x,z))^{Nr}} d\mu(z).
\end{equation}
Thus}, if $M_{\ell, t, N}f(x)  < \infty$, we conclude that
\begin{equation} \label{9035}
\big[M_{\ell,t, N}f(x)\big]^{r} \leq \widetilde{C}_{N} \sum_{j=0}^{\infty} 2^{-2Nrj} \int_{M}  \frac{|\omega_{j+\ell}(t^{1/2}\sqrt{\mathscr{L}})f(z)|^{r}}
{{|B(z,2^{-(j+\ell)})|}(1+2^{\ell}\rho(x,z))^{Nr}} d\mu(z).
\end{equation}

We claim that  for any $f \in \mathcal{S}_{\mathscr{L}}'$, there exists a positive number $N^{f}$ (depending on $f$)
such that $M_{\ell, t, N}f(x) < \infty$ for all $N > N^{f}$, $\ell \in  \mathbb{N}_0$ and $t \in [1,4]$.
Indeed, since $f$ is a linear functional on $\mathcal{S}_{\mathscr{L}}$,
there exist $k_{0}, \ell_{0} \in \mathbb{N}_{0}$ such that
\begin{align*}
\big|\omega_{i+\ell}(t^{1/2}\sqrt{\mathscr{L}})f(y)\big| &= \big| \big(f, K_{\omega_{i+\ell}(t^{1/2}\sqrt{\mathscr{L}})} (x,\cdot) \big) \big| \\
& \leq C_{f} \big\|K_{\omega_{i+\ell}(t^{1/2}\sqrt{\mathscr{L}})} (y,\cdot) \big\|_{\mathcal{P}_{k_{0}, \ell_{0}}} \\
& =C_{f} \sup_{z \in M} \big|K_{\mathscr{L}^{k_{0}}\omega_{i+\ell}(t^{1/2}\sqrt{\mathscr{L}}) } (y,z) \big| (1 +\rho(z, x_{0}))^{\ell_{0}}.
\end{align*}
Setting $\eta (\lambda) := \lambda^{2k_0} \omega (\lambda)$, by Lemma \ref{AOE}, \eqref{D} and \eqref{iuy} we have
\begin{align*}
\big|K_{\mathscr{L}^{k_{0}}\omega_{i+\ell}(t^{1/2}\sqrt{\mathscr{L}}) } (y,z) \big|
&=  (2^{(i+\ell)}t^{-1/2})^{2k_{0}}\big|K_{(2^{-(i +\ell)}t^{1/2}\sqrt{\mathscr{L}})^{2k_{0}}\omega_{i+\ell}(t^{1/2}\sqrt{\mathscr{L}}) } (y,z) \big| \\
&= \begin{cases}
 (2^{(i+\ell)}t^{-1/2})^{2k_{0}}\big|K_{\eta(2^{-(i+\ell)}t^{1/2}\sqrt{\mathscr{L}}) } (y,z) \big|, & i+\ell \geq 1, \\
    (t^{-1/2})^{2k_{0}}\big|K_{\omega_0(t^{1/2}\sqrt{\mathscr{L}}) } (y,z) \big|,          & i+\ell =0
\end{cases} \\
&\lesssim 2^{(i+\ell)2k_{0}} \big|B(y, 2^{-(i+\ell)}t^{1/2})\big|^{-1} \big(1 +2^{-(i + \ell)}t^{1/2}\rho(y,z)\big)^{-\ell_{0}} \\
&\lesssim 2^{(i+\ell)2k_{0}} 2^{(i+\ell)n}|B(y,1)|^{-1} 2^{(i + \ell) \ell_0}(1 +\rho(y,z))^{-\ell_{0}}\\
&\lesssim 2^{(i+\ell)(2k_{0}  +\ell_0 +n)} |B(x,1)|^{-1}(1 +\rho(x,y))^{n'}(1 +\rho(y,z))^{-\ell_{0}}.
\end{align*}
Hence, if $N \geq \max\big\{ k_{0} + \lfloor(\ell_0/2) +(n/2) \rfloor +1, \  \lfloor\ell_{0} + n' \rfloor +1\big\}:= N^{f}$, we have
{\small \begin{align*}
&M_{\ell,t, N}f(x)\\
&= \sup_{i \geq 0} \sup_{y \in M} 2^{-2Ni} \frac{\big|\omega_{i+\ell}(t^{1/2}
\sqrt{\mathscr{L}})f(y)\big|}{(1 +2^{\ell}\rho(x,y))^{N}} \\
& \leq C  \sup_{i \geq 0} \sup_{y \in M}\sup_{ z \in M} 2^{-2Ni}2^{(i+\ell)(2k_{0}  + \ell_0 +n)}
\frac{|B(x,1)|^{-1}(1 +\rho(x,y))^{n'}(1 + \rho(y, z))^{-\ell_{0}} (1 +\rho(z,x_{0}))^{\ell_{0}}}{  (1 +2^{\ell}\rho(x,y))^{N} } \\
&\leq C 2^{\ell (2k_{0} + \ell_0 +n)}|B(x,1)|^{-1}(1+\rho(x,x_{0}))^{\ell_{0}}\\
&< \infty.
\end{align*}
Thus} \eqref{9035} is valid provided that $N \geq N^{f}$.
This along with the obvious inequality $|\omega(2^{-\ell}t^{1/2}\sqrt{\mathscr{L}})f (x)| \leq M_{\ell,t, N}f(x)$ implies that for any $f \in \mathcal{S}_{\mathscr{L}}'$,
there exists a positive number $N^{f}$ (depending on $f$) such that if $N \geq N^{f}$ then for all $\ell \in \mathbb{N}_0$,
\begin{equation} \label{oiuy}
|\omega_\ell (t^{1/2}\sqrt{\mathscr{L}})f (x)|^{r}\leq c
\sum_{j=0}^{\infty} 2^{-2Nrj} \int_{M}  \frac{|\omega_{j+\ell}(t^{1/2}\sqrt{\mathscr{L}})f(z)|^{r}}
{{|B(z,2^{-(j+\ell)})|}(1+2^{\ell}\rho(x,z))^{Nr}} d\mu(z),
\end{equation}
where $c=\widetilde{C}_N$ is a constant depending on $N$ but independent of $x, f, t$ and $\ell$.
Observe that the right-hand side of \eqref{oiuy}
decreases as $N$ increases. Therefore, \eqref{oiuy}  is valid for all $N >0$ with the constant
\begin{align*}
c =C_{N,f} = \begin{cases}
\widetilde{C}_{N^{f}} & \mbox{if } 0 <N < N^{f} ,\\
\widetilde{C}_{N} & \mbox{if } N \geq N^{f}
\end{cases}
\end{align*}
depending on $N$ and $f$.

However, our goal is to obtain \eqref{oiuy} with $c$ independent of $f$. To do this,
let $N$ be an arbitrary positive integer. We may assume that the
right-hand side of \eqref{oiuy} is finite, for otherwise \eqref{oiuy} is trivial.
From \eqref{oiuy} with $c = C_{N,f}$ it follows that
\begin{align*}
\big[M_{\ell,t, N}f(x)\big]^{r}
& \leq C_{N,f}\sup_{i \geq 0} \sum_{j=0}^{\infty} 2^{-2Nr(j+i)} \int_{M}  \frac{|\omega_{j+i+\ell}(t^{1/2}\sqrt{\mathscr{L}})f(z)|^{r}}
{{|B(z,2^{-(j+i+\ell)})|}(1+2^{i+\ell}\rho(x,z))^{Nr}} d\mu(z) \\
& \leq C_{N,f} \sum_{j=0}^{\infty} 2^{-2Nrj} \int_{M}  \frac{|\omega_{j+\ell}(t^{1/2}\sqrt{\mathscr{L}})f(z)|^{r}}
{{|B(z,2^{-(j+\ell)})|}(1+2^{\ell}\rho(x,z))^{Nr}} d\mu(z)\\
& <\infty.
\end{align*}
Then \eqref{wlrt} along with the finiteness of $M_{\ell, t, N}f(x)$ for all $N\in \mathbb{N}$ yields \eqref{oiuy}
with the constant $c$ independent of $f$. Thus we have proved \eqref{oiuy} in the case $0<r \leq 1$.

Next we show \eqref{oiuy} for $r > 1$.  Indeed, we start with \eqref{xog}
(with $N+ \lfloor n +n'\rfloor +1$ instead of $N$), use H\"{o}lder's inequality first for the integrals
and then for the sums,  and apply \eqref{iuy} and \eqref{inte}, to obtain
{\small \begin{align*}
|\omega_\ell (t^{1/2}\sqrt{\mathscr{L}})f (x)|
& \leq \widetilde{C}_{N} \sum_{j=0}^{\infty} 2^{-2(N+ \lfloor n+n' \rfloor +1)j}
\int_{M}  \frac{|\omega_{j+\ell}(t^{1/2}\sqrt{\mathscr{L}})f(z)|}
{{|B(z,2^{-\ell})|}(1+2^{\ell}\rho(x,z))^{N+ \lfloor n+n' \rfloor +1}} d\mu(z)\\
& \leq \widetilde{C}_N \sum_{j=0}^{\infty}2^{-2(N+ \lfloor n+n' \rfloor +1)j}
 \left( \int_{M}  \frac{|\omega_{j+\ell}(t^{1/2}\sqrt{\mathscr{L}})f(z)|^{r}}
{{|B(z,2^{-\ell})|}(1+2^{\ell}\rho(x,z))^{Nr}} d\mu(z)\right)^{1/r}\\
&\quad\quad\quad\quad\quad\quad\quad \times  \left( \int_{M}  \frac{1}
{{|B(z,2^{-\ell})|}(1+2^{\ell}\rho(x,z))^{(\lfloor n+n' \rfloor +1)r'}} d\mu(z)\right)^{1/r'}\\
& \leq C_{r,N} \left( \sum_{j=0}^{\infty} 2^{-2Nrj}
\int_{M}  \frac{|\omega_{j+\ell}(t^{1/2}\sqrt{\mathscr{L}})f(z)|^{r}}
{{|B(z,2^{-\ell})|}(1+2^{\ell}\rho(x,z))^{Nr}} d\mu(z)\right)^{1/r},
\end{align*}
which} implies \eqref{oiuy} since $|B(z,2^{-\ell})| \geq |B(z,2^{-(j+\ell)})|$ for $j \geq 0$.

In summary, we have proved that \eqref{oiuy} holds for all $\ell \in \mathbb{N}_0$ and $t \in [1,4]$.
Obviously, \eqref{oiuy} covers \eqref{cen222}. However, it does not cover \eqref{cen333}.
Indeed, taking $\ell =0$ and $t =1$ in \eqref{oiuy}, we get
\begin{equation} \label{mmyy}
|\omega_0 (\sqrt{\mathscr{L}})f (x)|^{r}\lesssim
\sum_{j=0}^{\infty} 2^{-2Nrj} \int_{M}  \frac{|\omega_{j}(\sqrt{\mathscr{L}})f(z)|^{r}}
{{|B(z,2^{-j})|}(1+ \rho(x,z))^{Nr}} d\mu(z),
\end{equation}
which is different from \eqref{cen333}.

We thus need a separate argument to prove \eqref{cen333}. Fix an arbitrary $t_0 \in [1,4]$.
Let $\chi_0 (\lambda) := \omega_0 (\lambda)$ and $\chi(\lambda) := \omega(t_0^{1/2} \lambda)$,
and set $\chi_j (\lambda):= \chi(2^{-j}\lambda)$ for $j \geq 1$.
Then there exist functions $\widetilde{\psi}_0, \widetilde{\psi}$ which satisfy
similar properties as $\psi_0, \psi$ such that
\begin{equation*}
\sum_{j =0}^{\infty}\chi_j (\lambda)
\widetilde{\psi}_j (\lambda)=1, \quad \forall \lambda \in \mathbb{R},
\end{equation*}
where $\widetilde{\psi}_j (\lambda) := \widetilde{\psi}(2^{-j}\lambda)$ for $j \geq 1$.
Based on this identity, we can argue similarly as in the proof of \eqref{oiuy} to get
\begin{equation*}
|\chi_\ell (t^{1/2}\sqrt{\mathscr{L}})f (x)|^{r}\lesssim
\sum_{j=0}^{\infty} 2^{-2Nrj} \int_{M}  \frac{|\chi_{j+\ell}(t^{1/2}\sqrt{\mathscr{L}})f(z)|^{r}}
{{|B(z,2^{-(j+\ell)})|}(1+2^{\ell}\rho(x,z))^{Nr}} d\mu(z),
\end{equation*}
Letting $\ell =0$ and $t =1$, and recalling the definition of $\chi_0$ and $\chi_j$, we obtain
\begin{align*}
|\omega_0 (\sqrt{\mathscr{L}})f (x)|^{r} &\lesssim  \int_{M}  \frac{|\omega_0(\sqrt{\mathscr{L}})f(z)|^{r}}
{{|B(z,1)|}(1+ \rho(x,z))^{Nr}} d\mu(z)  \\
& \quad + \sum_{j=1}^{\infty} 2^{-2Nrj} \int_{M}  \frac{|\omega(2^{-j}t_0^{1/2}\sqrt{\mathscr{L}})f(z)|^{r}}
{{|B(z,2^{-j})|}(1+ \rho(x,z))^{Nr}} d\mu(z),
\end{align*}
with the implicit constant independent of $t_0$.  Since $t_0$ is an arbitrary number in $[1,4]$, the estimate \eqref{cen333} is established.
Thus the proof of Lemma \ref{moui} is complete.
\end{proof}

The following two lemmas also provide key ingredients of the proof of Theorem \ref{heat}.

\begin{lemma} \label{Lemma3}
Let $\varphi_0, \varphi$ be two even Schwartz functions on $\mathbb{R}$ satisfying \eqref{908} and \eqref{909}.
Let $s \in \mathbb{R}$ and $a >0$. Then there exists $\delta >0$ such that
\begin{equation} \label{hrq}
\begin{split}
&|B(x,2^{-\ell}t^{1/2})|^{-s/n}|\omega(2^{-\ell}t^{1/2}\sqrt{\mathscr{L}})f(x) | \\
&\lesssim 2^{-\ell \delta}\big[\varphi_{0}(\sqrt{\mathscr{L}}) \big]^{\ast}_{a, -s/n}f(x)
 + \sum_{j=1}^{\infty}2^{-|j -\ell|\delta}\big[\varphi(2^{-j}\sqrt{\mathscr{L}}) \big]^{\ast}_{a, -s/n}f(x), \quad \forall \ell \geq 1.
\end{split}
\end{equation}
and
\begin{equation} \label{inhomo}
\begin{split}
&|B(x, 1)|^{-s/n}| \omega_{0}(\sqrt{\mathscr{L}})f(x)| \\
&\lesssim \big[ \varphi_{0}(\sqrt{\mathscr{L}})\big]^{\ast}_{a,-s/n}f(x)
+ \sum_{j=1}^{\infty} 2^{-j\delta }\big[ \varphi(2^{-j}\sqrt{\mathscr{L}})\big]^{\ast}_{a,-s/n}f(x).
\end{split}
\end{equation}
\end{lemma}
\begin{proof}
In the proof of Lemma \ref{moui} we have seen that there exist even Schwartz functions $\psi_{0},\psi$ on $\mathbb{R}$ such that
$\supp \psi_{0} \subset \{|\lambda| \leq 2\}$, $\supp \psi \subset \{1/2 \leq |\lambda| \leq 2\}$, and
\begin{equation} \label{hpwwwww}
\omega_{0}({\lambda})\psi_{0}({\lambda})+ \sum_{j =1}^{\infty}\omega(2^{-j}{\lambda})
\psi(2^{-j}{\lambda})=1, \quad \forall \lambda \in \mathbb{R}.
\end{equation}
Then it follows from Lemma \ref{Calderon} that for any $f \in \mathcal{S}_{\mathscr{L}}'$,
\begin{align*}
f = \psi_{0}(\sqrt{\mathscr{L}})\varphi_{0}(\sqrt{\mathscr{L}})f
+ \sum_{j=1}^{\infty} \psi(2^{-j}\sqrt{\mathscr{L}})\varphi(2^{-j}\sqrt{\mathscr{L}})f \quad \mbox{in } \mathcal{S}_{\mathscr{L}}'.
\end{align*}
 Hence, for $\ell \geq 1$ and $t \in [1,4]$, we have
\begin{equation} \label{703}
\begin{split}
\omega(2^{-\ell}t^{1/2}\sqrt{\mathscr{L}})f(x)& = \omega(2^{-\ell}t^{1/2}\sqrt{\mathscr{L}})\psi_{0}(\sqrt{\mathscr{L}})\varphi_{0}(\sqrt{\mathscr{L}})f(x)\\
&\quad + \sum_{j=1}^{\infty}\omega(2^{-\ell}t^{1/2}\sqrt{\mathscr{L}}) \psi(2^{-j}\sqrt{\mathscr{L}})\varphi(2^{-j}\sqrt{\mathscr{L}})f(x)\\
&= \int_{M}K_{\omega(2^{-\ell}t^{1/2}\sqrt{\mathscr{L}})\psi_{0}(\sqrt{\mathscr{L}})}(x,y)\varphi_{0}(\sqrt{\mathscr{L}})f(y)d\mu(y)\\
&\quad + \sum_{j=1}^{\infty}\int_{M}K_{\omega(2^{-\ell}t^{1/2}\sqrt{\mathscr{L}}) \psi(2^{-j}\sqrt{\mathscr{L}})}(x,y)\varphi(2^{-j}\sqrt{\mathscr{L}})f(y)d\mu(y)
\end{split}
\end{equation}
and
\begin{equation} \label{hhha}
\begin{split}
\omega_{0}(\sqrt{\mathscr{L}})f (x)&= \omega_{0}(\sqrt{\mathscr{L}})\psi_{0}(\sqrt{\mathscr{L}})\varphi_{0}(\sqrt{\mathscr{L}})(x) \\
&\quad + \sum_{j=1}^{\infty}\omega_{0}(\sqrt{\mathscr{L}}) \psi(2^{-j}\sqrt{\mathscr{L}})\varphi(2^{-j}\sqrt{\mathscr{L}})f(x)\\
&= \int_{M}K_{\omega_{0}(\sqrt{\mathscr{L}})\psi_{0}(\sqrt{\mathscr{L}})}(x,y)\varphi_{0}(\sqrt{\mathscr{L}})(y)d\mu(y) \\
&\quad +   \sum_{j=1}^{\infty}\int_{M}K_{\omega_{0}(\sqrt{\mathscr{L}}) \psi(2^{-j}\sqrt{\mathscr{L}})}(x,y)\varphi(2^{-j}\sqrt{\mathscr{L}})f(y)d\mu(y).
\end{split}
\end{equation}

Let $N \in \mathbb{N}$ such that $N -n -n' - |s|n'/n -a > n + n'/2$, and let $m' \in \mathbb{N}$
such that $2m' -\max\{-s, 0\} -a >0$.
If $\ell \geq 1$ and  $j \leq \ell$, then
since $(\cdot)^{-2m}\omega(t^{1/2}\cdot) \in \mathcal{S}(\mathbb{R})$ and
\begin{align*}
\sup_{t \in [1,4]} \big\|(\cdot)^{-2m}\omega(t^{1/2}\cdot)\big\|_{(N)} \leq c_{N},
\end{align*}
by Lemma \ref{AOE} we have
\begin{align} \label{hqy111}
\big| K_{\omega(2^{-\ell}t^{1/2}\sqrt{\mathscr{L}}) \psi(2^{-j}\sqrt{\mathscr{L}})}(x,y) \big|\lesssim 2^{-2m(\ell -j)}D_{2^{-j},N-n-n'}(x,y),
\quad \forall t \in [1,4].
\end{align}

If $\ell \geq 1$ and $j \geq \ell$, then using
the fact that $(\cdot)^{-2m'}\psi(\cdot) \in \mathcal{S}(\mathbb{R})$ (since $\psi$ vanishes near the origin) and Lemma \ref{AOE} we get
\begin{align} \label{hqy222}
\big| K_{\omega(2^{-\ell}t^{1/2}\sqrt{\mathscr{L}}) \psi(2^{-j}\sqrt{\mathscr{L}})}(x,y) \big|\lesssim 2^{-2m'(j -\ell)}D_{2^{-\ell},N-n-n'}(x,y),
\quad \forall t \in [1,4].
\end{align}

Also, for all $\ell \geq 1$ we have
\begin{align} \label{hqy333}
\big|K_{\omega(2^{-\ell}t^{1/2}\sqrt{\mathscr{L}})\psi_{0}(\sqrt{\mathscr{L}})}(x,y)\big| \lesssim 2^{-2m\ell}D_{1, N-n-n'}(x,y), \quad
\forall t \in [1,4].
\end{align}

Inserting the estimates \eqref{hqy111}, \eqref{hqy222} and \eqref{hqy333} into \eqref{703},  we obtain
that for all $\ell \geq 1$ and $t \in [1,4]$,
\begin{equation} \label{8911}
\begin{split}
&|B(x,2^{-\ell}t^{1/2})|^{-s/n}|\omega(2^{-\ell}t^{1/2}\sqrt{\mathscr{L}})f(x) | \\
& \lesssim 2^{-2m\ell}\int_{M}|B(x,2^{-\ell}t^{1/2})|^{-s/n} |\varphi_{0}(\sqrt{\mathscr{L}})f(y)|D_{1, N-n-n'}(x,y)d\mu(y) \\
&+ \sum_{j=1}^{\ell}2^{-2m(\ell -j)}\int_{M}|B(x,2^{-\ell}t^{1/2})|^{-s/n}|\varphi(2^{-j}\sqrt{\mathscr{L}})f(y)|D_{2^{-j},N-n-n'}(x,y)d\mu(y)\\
& + \sum_{j=\ell +1}^{\infty}2^{-2m'(j-\ell)}\int_{M}|B(x,2^{-\ell}t^{1/2})|^{-s/n}|\varphi(2^{-j}\sqrt{\mathscr{L}})f(y)|D_{2^{-\ell},N-n-n'}(x,y)d\mu(y).
\end{split}
\end{equation}

If $0 \leq j \leq \ell$, we have
\begin{align*}
|B(x, 2^{-\ell}t^{1/2})|^{-s/n} &\lesssim 2^{(\ell -j) \max \{s,0\}} |B(x,2^{-j} )|^{-s/n} \\
&\lesssim 2^{(\ell -j) \max \{s,0\}}|B(y,2^{-j})|^{-s/n}| (1 + 2^{j}\rho(x,y))^{|s|n'/n},
\end{align*}
while if $j \geq \ell +1$, we have
\begin{align*}
|B(x,2^{-\ell}t^{1/2})|^{-s/n}& \lesssim |B(y,2^{-\ell} )|^{-s/n}(1 +2^{\ell}\rho(x,y))^{|s|n'/n} \\
& \lesssim 2^{(j -\ell)\max\{-s,0\}}|B(y,2^{-j})|^{-s/n}(1 +2^{\ell}\rho(x,y))^{|s|n'/n} .
\end{align*}
Using these facts, we deduce from \eqref{8911} that for all $\ell \geq 1$ and $t \in [1,4]$,
\begin{align*}
& |B(x,2^{-\ell}t^{1/2})|^{-s/n}|\omega(2^{-\ell}t^{1/2}\sqrt{\mathscr{L}})f(x) |  \\
& \lesssim 2^{-(2m-\max\{s,0\})\ell} \int_{M}|B(y,1)|^{-s/n}|\varphi_{0}(\sqrt{\mathscr{L}})f(y)|D_{1,N-n-n'-|s|n'/n}(x,y)d\mu(y)  \\
&\quad  + \sum_{j=1}^{\ell}2^{-2m(\ell-j)}2^{(\ell -j)\max\{s,0\}}\int_{M}|B(y,2^{-j})|^{-s/n} |\varphi(2^{-j}\sqrt{\mathscr{L}})f(y)| \\
&\quad\quad\quad\quad\quad\quad\quad\quad\quad\quad\quad\quad \quad\quad\quad\quad
\times D_{2^{-j},N-n-n'-|s|n'/n}(x,y)d\mu(y)   \\
&\quad + \sum_{j=\ell +1}^{\infty}2^{-2m'(j-\ell)}2^{(j-\ell)\max\{-s,0\}}\int_{M}|B(y,2^{-j})|^{-s/n} |\varphi(2^{-j}\sqrt{\mathscr{L}})f(y)| \\
&\quad\quad\quad\quad\quad\quad\quad\quad\quad\quad\quad\quad \quad\quad\quad\quad \quad
\times D_{2^{-\ell },N-n-n'-|s|n'/n}(x,y)d\mu(y) \\
& \lesssim 2^{-(2m-\max\{s,0\})\ell} \int_{M}\frac{|B(y,1)|^{-s/n}|\varphi_{0}(\sqrt{\mathscr{L}})f(y)|}{(1 + \rho(x,y))^a}D_{1,N-n-n'-|s|n'/n-a}(x,y)d\mu(y)\\
&\quad  + \sum_{j=1}^{\ell}2^{-2m(\ell-j)}2^{(\ell -j)\max\{s,0\}}\int_{M}\frac{|B(y,2^{-j})|^{-s/n} |\varphi(2^{-j}\sqrt{\mathscr{L}})f(y)|}{(1 + 2^j \rho(x,y))^a}\\
&\quad\quad\quad\quad\quad\quad\quad\quad\quad\quad\quad\quad \quad\quad\quad
\times D_{2^{-j},N-n-n'-|s|n'/n -a}(x,y)d\mu(y) \nonumber  \\
&\quad + \sum_{j=\ell +1}^{\infty}2^{-2m'(j-\ell)}2^{(j-\ell)\max\{-s,0\}}2^{(j-\ell)a}\int_{M}\frac{|B(y,2^{-j})|^{-s/n} |\varphi(2^{-j}\sqrt{\mathscr{L}})f(y)|}
{(1 + 2^j\rho(x,y))^a} \\
&\quad\quad\quad\quad\quad\quad\quad\quad\quad\quad\quad\quad \quad\quad\quad\quad\quad\quad\quad
\times D_{2^{-\ell },N-n-n'-|s|n'/n-a}(x,y)d\mu(y)    \\
&\lesssim 2^{-(2m-\max\{s,0\})\ell}\big[\varphi_{0}(\sqrt{\mathscr{L}}) \big]^{\ast}_{a}f(x)+ \sum_{j=1}^{\ell}2^{-(2m-\max\{s,0\})(\ell -j)}\big[\varphi(2^{-j}\sqrt{\mathscr{L}}) \big]^{\ast}_{a}f(x) \\
& \quad + \sum_{j=\ell +1}^{\infty}2^{-(2m'-\max\{-s,0\}-a)(j -\ell)}\big[\varphi(2^{-j}\sqrt{\mathscr{L}}) \big]^{\ast}_{a}f(x) .
\end{align*}
where we used the elementary inequality
\begin{equation*}
(1 + 2^{j}\rho(x,y))^{a} \lesssim 2^{(j-\ell)a}(1 +2^{\ell}\rho(x,y))^{a} \quad \mbox{for } j \geq \ell,
\end{equation*}
and \eqref{ges} (taking into account that $N -n -n' - |s|n'/n -a > n + n'/2$).
This gives \eqref{hrq} with $\delta= \min\big\{2m-\max\{s,0\}, \  2m'-\max\{-s, 0\} - a\big\} >0$.

In a similar manner we can deduce \eqref{inhomo} from \eqref{hhha}.
The proof of Lemma \ref{Lemma3} is thus complete.
\end{proof}

\begin{lemma} \label{Lemma2}
Let $\varphi_0, \varphi$ be two even Schwartz functions on $\mathbb{R}$ satisfying \eqref{908} and \eqref{909}.
Let $s \in \mathbb{R}$ and $a >0$. Then there exists $\delta >0$ such that
\begin{equation} \label{5652}
\begin{split}
&|B(x,2^{-\ell})|^{-s/n}\big|\varphi(2^{-\ell}\sqrt{\mathscr{L}})f(x)\big|\\
 &\lesssim  2^{-\ell \delta } \big[\omega_{0}(\sqrt{\mathscr{L}})\big]^{\ast}_{a, -s/n}f(x)   + \sum_{j=1}^{\infty} 2^{- |j -\ell|\delta}\big[\omega(2^{-j}t^{1/2}\sqrt{\mathscr{L}})\big]^{\ast}_{a, -s/n}f(x), \quad \forall \ell \geq 1,
\end{split}
\end{equation}
and
\begin{equation} \label{traaa}
\begin{split}
&|B(x,1)|^{-s/n}\big|\varphi_{0}(\sqrt{\mathscr{L}})f(x)\big| \\
& \lesssim  \big[\omega_{0}(\sqrt{\mathscr{L}}) \big]^{\ast}_{a, -s/n}f(x)
+ \sum_{j=1}^{\infty} 2^{-j \delta} \big[ \omega(2^{-j}t^{1/2}\sqrt{\mathscr{L}})\big]^{\ast}_{a, -s/n}f(x).
\end{split}
\end{equation}
\end{lemma}
\begin{proof}
As before there exist
even Schwartz functions $\psi_{0},\psi$ on $\mathbb{R}$ such that
$\supp \psi_{0} \subset \{\lambda \in \mathbb{R}: |\lambda| \leq 2\}$, $\supp \psi \subset
\{\lambda \in \mathbb{R}: 1/2 \leq |\lambda| \leq 2\}$, and
\begin{align} \label{rig}
\psi_{0}(\lambda)\omega_{0}(\lambda)+\sum_{j =1}^{\infty} \psi(2^{-j}\lambda) \omega(2^{-j} \lambda) =1, \quad \forall \lambda \in \mathbb{R}.
\end{align}
Define
\begin{align} \label{rig22}
\theta(\lambda): = 1 -\sum_{j=1}^{\infty} \psi(2^{-j}\lambda) \omega(2^{-j}\lambda) = \psi_{0}(\lambda) \omega_{0}(\lambda), \quad \lambda \in \mathbb{R}.
\end{align}
Observe that $\supp \theta \subset \{\lambda \in \mathbb{R}: |\lambda| \leq 2\}$.
Since $\omega_{0}(\lambda) >0$ on $\{\lambda \in \mathbb{R}: |\lambda| \leq 2\}$, there exists an
even function $\phi \in \mathcal{S}(\mathbb{R})$ such that
\begin{align} \label{hy}
\phi(\lambda) = \frac{1}{\omega_{0}(\lambda)} \quad \mbox{when } |\lambda| \leq 2.
\end{align}
We now set
\begin{align*}
\psi_{0,t}(\lambda):=\phi(\lambda) \theta(t^{1/2}\lambda), \quad  t \in [1,4], \  \lambda\in\mathbb{R}.
\end{align*}
Then the fact that $\supp \theta (t \cdot) \subset \{\lambda \in \mathbb{R}:|\lambda| \leq 2\}$ ($\forall t \in [1,4]$)
along with \eqref{hy} implies that
\begin{align*}
\psi_{0,t}(\lambda) \omega_{0}(\lambda) = \phi(\lambda) \theta(t^{1/2}\lambda) \omega_{0}(\lambda)=\theta(t^{1/2}\lambda),
\quad \forall t \in [1,4], \ \forall \lambda\in\mathbb{R}.
\end{align*}
From this, \eqref{rig} and \eqref{rig22}, we deduce that
\begin{align*}
\psi_{0,t}(\lambda)\omega_{0}(\lambda) + \sum_{j=1}^{\infty} \psi(2^{-j}t^{1/2}\lambda) \omega(2^{-j}t^{1/2}
\lambda) =1, \quad \forall t \in [1,4], \
\forall \lambda \in \mathbb{R}.
\end{align*}
Hence it follows from Lemma \ref{Calderon} that for any $f \in \mathcal{S}_{\mathscr{L}}'$,
\begin{align} \label{nontrivial}
f = \psi_{0,t}(\sqrt{\mathscr{L}})\omega_{0}(\sqrt{\mathscr{L}})f +  \sum_{j=1}^{\infty}
\psi(2^{-j}t^{1/2}\sqrt{\mathscr{L}}) \omega(2^{-j}t^{1/2}\sqrt{\mathscr{L}})f
\quad \mbox{in } \mathcal{S}_{\mathscr{L}}'.
\end{align}
Consequently, for $\ell \geq 1$ and $y \in M$, we have
\begin{align} \label{1455}
\begin{split}
\varphi(2^{-\ell}\sqrt{\mathscr{L}})f(x) &= \varphi(2^{-\ell}\sqrt{\mathscr{L}})\psi_{0,t}(\sqrt{\mathscr{L}})\omega_{0}(\sqrt{\mathscr{L}})f(x) \\
& \quad+ \sum_{j=1}^{\infty} \varphi(2^{-\ell}\sqrt{\mathscr{L}})\psi(2^{-j}t^{1/2}\sqrt{\mathscr{L}}) \omega(2^{-j}t^{1/2}\sqrt{\mathscr{L}})f(x)\\
&= \int_{M} K_{\varphi(2^{-\ell}\sqrt{\mathscr{L}})\psi_{0,t}(\sqrt{\mathscr{L}})}(x,y)\omega_{0}(\sqrt{\mathscr{L}})f(y)d\mu(y) \\
& \quad + \sum_{j=1}^{\infty} \int_{M}K_{ \varphi(2^{-\ell}\sqrt{\mathscr{L}})\psi(2^{-j}t^{1/2}\sqrt{\mathscr{L}})} (x,y) \omega(2^{-j}t^{1/2}\sqrt{\mathscr{L}})f(y)d\mu(y).
\end{split}
\end{align}
Analogously,
\begin{equation} \label{omag}
\begin{split}
\varphi_{0}(\sqrt{\mathscr{L}})f(x)
&= \int_{M}K_{\varphi_{0}(\sqrt{\mathscr{L}})\psi_{0,t}(\sqrt{\mathscr{L}})}(x,y)\omega_{0}(\sqrt{\mathscr{L}})f (y)d\mu(y)\\
&\quad + \sum_{j=1}^{\infty} \int_{M}K_{\varphi_{0}(\sqrt{\mathscr{L}})\psi(2^{-j}t^{1/2}\sqrt{\mathscr{L}})}(x,y) \omega(2^{-j}t^{1/2}\sqrt{\mathscr{L}})f(y)d\mu(y).
\end{split}
\end{equation}
Since both $\varphi$ and $\psi$ vanish near the origin, we have
$(\cdot)^{-2k} \phi(\cdot)
, \ (\cdot)^{-2k} \psi(\cdot)
\in \mathcal{S}(\mathbb{R})$ for arbitrarily large $k \in \mathbb{N}$.
Hence by Lemma \ref{AOE} we have sufficiently good estimates for the kernels
in \eqref{1455} and \eqref{omag}.
Therefore we can argue similarly as in the proof of Lemma~\ref{Lemma3} to obtain the desired estimates \eqref{5652}
and \eqref{traaa}.
\end{proof}

We are now in position to prove Theorem \ref{heat}.
\begin{proof}[Proof of Theorem \ref{heat}]
We shall only give the proofs of \eqref{hqy1} and \eqref{hqy2} for ``nonclassical'' spaces;
the proofs of \eqref{hqy00} and \eqref{hqy000} for ``classical'' spaces are easier and will be omitted.

In what follows, let $\varphi_0, \varphi$ be even Schwartz functions on $\mathbb{R}$ satisfying \eqref{908} and \eqref{909}.

\medskip
{\it Step 1.} We are going to show that for any $a >0$ and $f \in \mathcal{S}_{\mathscr{L}}'$,
\begin{equation} \label{pmg1}
\begin{split}
& \big\||B(\cdot, 1)|^{-s/n}\varphi_{0}(\sqrt{\mathscr{L}})
f(\cdot)\big\|_{L^{p}} + \left( \sum_{j =1}^{\infty} \big\| |B(\cdot, 2^{-j})|^{-s/n}\varphi
(2^{-j}\sqrt{\mathscr{L}})f(\cdot)\big\|_{L^{p}}^{q}
  \right)^{1/q}  \\
& \lesssim \left\|\big[\omega_{0}(\sqrt{\mathscr{L}})\big]^{\ast}_{a, -s/n}f \right\|_{L^{p}} +  \left( \int_{0}^{1}
 \left\|\big[ \omega(t^{1/2}\sqrt{\mathscr{L}})\big]^{\ast}_{a, -s/n}f\right\|_{L^{p}}^{q}\frac{dt}{t}\right)^{1/q}
\end{split}
\end{equation}
and
\begin{equation} \label{pmg}
\begin{split}
&\big\|||B(\cdot, 1)|^{-s/n}\varphi_{0}(\sqrt{\mathscr{L}})f(\cdot)| \big\|_{L^{p}} +\left\|\left(\sum_{j= 1}^{\infty}
|B(\cdot, 2^{-j})|^{-sq/n} |\varphi(2^{-j}\sqrt{\mathscr{L}})f(\cdot)|^{q}\right)^{1/q}\right\|_{L^{p}} \\
& \lesssim \left\|\big[\omega_{0}(\sqrt{\mathscr{L}})\big]^{\ast}_{a, -s/n}f \right\|_{L^{p}} + \left\|  \left( \int_{0}^{1}
\left|\big[ \omega(t^{1/2}\sqrt{\mathscr{L}})\big]^{\ast}_{a, -s/n}f\right|^{q}\frac{dt}{t}\right)^{1/q}\right\|_{L^{p}}.
\end{split}
\end{equation}

We first prove \eqref{pmg1}. To do this, we consider two cases.

{\bf Case 1:} $1 < q <\infty$. We take the norm $\big(\int_{1}^{4}|\cdot|^{q}\frac{dt}{t}\big)^{1/q}$ on both sides of \eqref{5652} and \eqref{traaa}
in Lemma \ref{Lemma2}
(noting that the left-hand sides of these inequalities are independent of $t$), and use the Minkowski's inequality to obtain
{\small \begin{equation} \label{zh}
\begin{split}
& |B(x,2^{-\ell})|^{-s/n}\big|\varphi(2^{-\ell}\sqrt{\mathscr{L}})f(x)\big| \\
&\lesssim 2^{- \ell \delta} \big[\omega_{0}(\sqrt{\mathscr{L}}) \big]^{\ast}_{a, -s/n}f(x)  + \sum_{j=1}^{\infty} 2^{-|j-\ell| \delta}\left(\int_{1}^{4}
\left|\big[ \omega(2^{-j}t^{1/2}\sqrt{\mathscr{L}})\big]^{\ast}_{a, -s/n}f(x)\right|^{q}\frac{dt}{t}\right)^{1/q}
\end{split}
\end{equation}
for} all $\ell \geq 1$, and
\begin{equation} \label{xia}
\begin{split}
&|B(x,1)|^{-s/n}\big|\varphi_{0}(\sqrt{\mathscr{L}})f(x)\big|\\
& \lesssim  \big[\omega_{0}(\sqrt{\mathscr{L}}) \big]^{\ast}_{a, -s/n}f(x) + \sum_{j=1}^{\infty} 2^{-j \delta}\left(\int_{1}^{4}
\left|\big[ \omega(2^{-j}t^{1/2}\sqrt{\mathscr{L}})\big]^{\ast}_{a, -s/n}f(x)\right|^{q}\frac{dt}{t}\right)^{1/q}.
\end{split}
\end{equation}
Put
\begin{align*}
G_{\ell}(x):= \begin{cases}
|B(x,1)|^{-s/n}\big|\varphi_{0}(\sqrt{\mathscr{L}})f(x) , & \ell =0 ,   \\
|B(x,2^{-\ell})|^{-s/n}
\big|\varphi(2^{-\ell}\sqrt{\mathscr{L}})f(x)\big|, &\ell =1,2,\cdots
\end{cases}
\end{align*}
and
\begin{align*}
g_j (x):= \begin{cases}
\big[\omega_{0}(\sqrt{\mathscr{L}})\big]^{\ast}_{a, -s/n}f(x), & j =0  ,  \\
\displaystyle \left(\int_{1}^{4}\big|\big[ \omega(2^{-j}t^{1/2}\sqrt{\mathscr{L}})\big]^{\ast}_{a, -s/n}f(x)\big|^{q}
\frac{dt}{t}\right)^{1/q}, &j =1,2,\cdots.
\end{cases}
\end{align*}
Then \eqref{zh} coupled with \eqref{xia} implies
\begin{equation*}
G_{\ell}(x) \lesssim \sum_{j=0}^{\infty} 2^{-|j -\ell|\delta} g_{j}(x), \quad \ell =0,1,2,\cdots,
\end{equation*}
Applying Lemma \ref{sfga} in $L^{p}(\ell^{q})$ then yields
\begin{align*}
&\big\||B(\cdot, 1)|^{-s/n}\varphi_{0}(\sqrt{\mathscr{L}})f(\cdot) \big\|_{L^{p}} +\left\|\left(\sum_{j= 1}^{\infty}
|B(\cdot, 2^{-j})|^{-sq/n} |\varphi(2^{-j}\sqrt{\mathscr{L}})f(\cdot)|^{q}\right)^{1/q}\right\|_{L^{p}} \\
& \lesssim  \left\|\big[\omega_{0}(\sqrt{\mathscr{L}})\big]^{\ast}_{a, -s/n}f \right\|_{L^{p}} +\left\|\left(\sum_{j=1}^{\infty}\int_{1}^{4}\left|\big[ \omega(2^{-j}t^{1/2}\sqrt{\mathscr{L}})\big]^{\ast}_{a, -s/n}f\right|^{q}\frac{dt}{t}\right)^{1/q}\right\|_{L^{p}} \\
& = \left\|\big[\omega_{0}(\sqrt{\mathscr{L}})\big]^{\ast}_{a, -s/n}f \right\|_{L^{p}} +\left\|\left(\sum_{j=1}^{\infty}\int_{2^{-2j}}^{2^{-2(j-1)}}\left|\big[ \omega(t^{1/2}\sqrt{\mathscr{L}})\big]^{\ast}_{a, -s/n}f\right|^{q}\frac{dt}{t}\right)^{1/q}\right\|_{L^{p}} \\
& = \left\|\big[\omega_{0}(\sqrt{\mathscr{L}})\big]^{\ast}_{a, -s/n}f \right\|_{L^{p}} +\left\|\left(\int_{0}^{1}\left|\big[ \omega(t^{1/2}\sqrt{\mathscr{L}})\big]^{\ast}_{a, -s/n}f\right|^{q}\frac{dt}{t}\right)^{1/q}\right\|_{L^{p}} .
\end{align*}
This proves \eqref{pmg1} in the case $q >1$.

{\bf Case 2:} $0< q  \leq 1$. Using the inequality $(\sum_j u_j)^q \leq \sum_j |u_j|^q$, we deduce from \eqref{5652} and \eqref{traaa} that
{\small \begin{equation} \label{zh2}
\begin{split}
& |B(x,2^{-\ell})|^{-sq/n}\big|\varphi(2^{-\ell}\sqrt{\mathscr{L}})f(x)\big|^{q} \\
& \lesssim 2^{-\ell \delta q} \left| \big[\omega_{0}(\sqrt{\mathscr{L}}) \big]^{\ast}_{a, -s/n}f(x)\right|^{q}
 + \sum_{j=1}^{\infty} 2^{-|j-\ell| \delta q}\int_{1}^{4} \left|\big[ \omega(2^{-j}t^{1/2}\sqrt{\mathscr{L}})\big]^{\ast}_{a, -s/n}f(x)\right|^{q}\frac{dt}{t}
\end{split}
\end{equation}}
for all $\ell \geq 1$, and
\begin{equation} \label{xia2}
\begin{split}
&|B(x,1)|^{-sq/n}\big|\varphi_{0}(\sqrt{\mathscr{L}})f(x)\big|^{q} \\
& \lesssim \left| \big[\omega_{0}(\sqrt{\mathscr{L}}) \big]^{\ast}_{a, -s/n}f(x)\right|^{q}
+ \sum_{j=1}^{\infty} 2^{-j \delta q}\int_{1}^{4} \left|\big[ \omega(2^{-j}t^{1/2}\sqrt{\mathscr{L}})\big]^{\ast}_{a, -s/n}f(x)\right|^{q}\frac{dt}{t}.
\end{split}
\end{equation}
Put
 \begin{align*}
 \widetilde{G}_{\ell}(x):= \begin{cases}
 |B(x,1)|^{-sq/n}\big|\varphi_{0}(\sqrt{\mathscr{L}})f(x)\big|^{q}, & \ell =0, \\
  |B(x,2^{-\ell})|^{-sq/n}\big|\varphi(2^{-\ell}\sqrt{\mathscr{L}})f(x)\big|^{q}, & \ell =1,2, \cdots
 \end{cases}
 \end{align*}
 and
 \begin{align*}
\widetilde{g}_{j}(x):= \begin{cases}
 \left|\big[\omega_{0}(\sqrt{\mathscr{L}})\big]^{\ast}_{a, -s/n}f(x) \right|^{q}, & j =0, \\
  \displaystyle \int_{1}^{4} \left|\big[ \omega(2^{-j}t^{1/2}\sqrt{\mathscr{L}})\big]^{\ast}_{a, -s/n}f(x)
  \right|^{q}\frac{dt}{t}, & j =1,2, \cdots.
 \end{cases}
 \end{align*}
Then \eqref{zh2} coupled with \eqref{xia2} implies
\begin{equation} \label{0683}
\widetilde{G}_{\ell}(x) \lesssim \sum_{j=0}^{\infty} 2^{-|j -\ell|\delta q } \widetilde{g}_{j}(x), \quad \ell =0,1,2,\cdots,
\end{equation}
 Applying Lemma \ref{sfga} in $L^{p/q}(\ell^{1})$ then yields
\begin{align*}
&\big\||B(\cdot, 1)|^{-s/n}\varphi_{0}(\sqrt{\mathscr{L}})f(\cdot) \big\|_{L^{p}} +\left\|\left(\sum_{j= 1}^{\infty}
|B(\cdot, 2^{-j})|^{-sq/n} |\varphi(2^{-j}\sqrt{\mathscr{L}})f(\cdot)|^{q}\right)^{1/q}\right\|_{L^{p}}\\
&=\big\||B(\cdot, 1)|^{-sq/n}|\varphi_{0}(\sqrt{\mathscr{L}})f(\cdot)|^q \big\|_{L^{p/q}}^{1/q} +\left\|\sum_{j= 1}^{\infty}
|B(\cdot, 2^{-j})|^{-sq/n} |\varphi(2^{-j}\sqrt{\mathscr{L}})f(\cdot)|^{q}\right\|_{L^{p/q}}^{1/q} \\
& \lesssim  \left\|\big|\big[\omega_{0}(\sqrt{\mathscr{L}})\big]^{\ast}_{a, -s/n}f\big|^q \right\|_{L^{p/q}}^{1/q} +\left\|\sum_{j=1}^{\infty}\int_{1}^{4}\left|\big[ \omega(2^{-j}t^{1/2}\sqrt{\mathscr{L}})\big]^{\ast}_{a, -s/n}f(x)\right|^{q}\frac{dt}{t}\right\|_{L^{p/q}}^{1/q} \\
& =  \left\|\big|\big[\omega_{0}(\sqrt{\mathscr{L}})\big]^{\ast}_{a, -s/n}f\big|^q \right\|_{L^{p/q}}^{1/q} +\left\|\sum_{j=1}^{\infty}\int_{2^{-2j}}^{2^{-2(j-1)}}\left|\big[ \omega(t^{1/2}\sqrt{\mathscr{L}})\big]^{\ast}_{a, -s/n}f(x)\right|^{q}\frac{dt}{t}\right\|_{L^{p/q}}^{1/q} \\
& = \left\|\big|\big[\omega_{0}(\sqrt{\mathscr{L}})\big]^{\ast}_{a, -s/n}f\big|^q \right\|_{L^{p/q}}^{1/q} +\left\|\int_{0}^{1}\left|\big[ \omega(t^{1/2}\sqrt{\mathscr{L}})\big]^{\ast}_{a, -s/n}f(x)\right|^{q}\frac{dt}{t}\right\|_{L^{p/q}}^{1/q} \\
& = \left\|\big[\omega_{0}(\sqrt{\mathscr{L}})\big]^{\ast}_{a, -s/n}f \right\|_{L^{p}} +\left\|\left(\int_{0}^{1}\left|\big[ \omega(t^{1/2}\sqrt{\mathscr{L}})\big]^{\ast}_{a, -s/n}f(x)\right|^{q}\frac{dt}{t}\right)^{1/q}\right\|_{L^{p}}.
\end{align*}
This proves \eqref{pmg1} in the case $0< q \leq 1$.

\medskip
Next we prove \eqref{pmg}.  We still consider two cases.

{\bf Case A: } $1 < p \leq \infty$. In this case, we take the $L^p$ norm on both sides of \eqref{hqq} and \eqref{hqqq} and use Minkowski's inequality to get
\begin{equation} \label{hugg}
\begin{split}
& \left\||B(\cdot,2^{-\ell})|^{-s/n} \varphi(2^{-\ell}\sqrt{\mathscr{L}})f(\cdot)\right\|_{L^{p}} \\
 &\lesssim  2^{-\ell \delta} \left\|\big[\omega_{0}(\sqrt{\mathscr{L}})\big]^{\ast}_{a, -s/n}f\right\|_{L^p} + \sum_{j=1}^{\infty} 2^{-|j -\ell|\delta}\left\|\big[\omega(2^{-j}t^{1/2}\sqrt{\mathscr{L}})\big]^{\ast}_{a, -s/n}f\right\|_{L^p}, \quad \ell \geq 1
\end{split}
\end{equation}
and
\begin{equation} \label{huggg}
\begin{split}
& \big\||B(\cdot,1)|^{-s/n} \varphi_0(\sqrt{\mathscr{L}})f(\cdot)\big\|_{L^{p}} \\
 &\lesssim  \left\|\big[\omega_{0}(\sqrt{\mathscr{L}})\big]^{\ast}_{a, -s/n}f\right\|_{L^p} + \sum_{j=1}^{\infty} 2^{- j \delta}\left\|\big[\omega(2^{-j}t^{1/2}\sqrt{\mathscr{L}})\big]^{\ast}_{a, -s/n}f\right\|_{L^p}.
\end{split}
\end{equation}
Applying Lemma \ref{sfga} with respect to the variable $t$ (and noting that the left-hand sides of \eqref{hugg} and \eqref{huggg} are independent of $t$),
we then obtain
\begin{align*}
& \big\||B(\cdot,1)|^{-s/n} \varphi_0(\sqrt{\mathscr{L}})f(\cdot)\big\|_{L^{p}(M)}  + \left\| \left\{  \big\||B(\cdot,2^{-j})|^{-s/n} \varphi(2^{-j}\sqrt{\mathscr{L}})f(\cdot)\big\|_{L^{p}(M)}\right\}_{j=1}^\infty \right\|_{\ell^q} \\
& \lesssim \left\|\big[\omega_{0}(\sqrt{\mathscr{L}})\big]^{\ast}_{a, -s/n}f\right\|_{L^p(M)} +\left\|\left\{\left\| \left\|\big[\omega(2^{-j}t^{1/2}\sqrt{\mathscr{L}})\big]^{\ast}_{a, -s/n}f\right\|_{L^p(M)}\right\|_{L^q([1,4],dt/t)}  \right\}_{j=1}^{\infty} \right\|_{\ell^q}
\end{align*}
This gives \eqref{pmg} by direct computation.

{\bf Case B.} $0< q \leq 1$. Rasing \eqref{5652} and \eqref{traaa} to the power $p$, using the inequality $(\sum_j u_j )^p \leq \sum_j |u_j|^p$,
and then integrating both sides on $M$, we obtain
{\small \begin{equation} \label{hqq}
\begin{split}
& \big\||B(\cdot,2^{-\ell})|^{-s/n} \varphi(2^{-\ell}\sqrt{\mathscr{L}})f(\cdot)\big\|_{L^{p}}^p \\
 &\lesssim  2^{-\ell \delta p} \left\|\big[\omega_{0}(\sqrt{\mathscr{L}})\big]^{\ast}_{a, -s/n}f\right\|_{L^p}^p + \sum_{j=1}^{\infty} 2^{-|j -\ell| \delta p}\left\|\big[\omega(2^{-j}t^{1/2}\sqrt{\mathscr{L}})\big]^{\ast}_{a, -s/n}f\right\|_{L^p}^p, \quad \ell \geq 1,
\end{split}
\end{equation}
and}
\begin{equation} \label{hqqq}
\begin{split}
& \big\||B(\cdot,1)|^{-s/n} \varphi_0(\sqrt{\mathscr{L}})f(\cdot)\big\|_{L^{p}}^p \\
 &\lesssim   \left\|\big[\omega_{0}(\sqrt{\mathscr{L}})\big]^{\ast}_{a, -s/n}f\right\|_{L^p}^p + \sum_{j=1}^{\infty} 2^{-j  \delta p}\left\|\big[\omega(2^{-j}t^{1/2}\sqrt{\mathscr{L}})\big]^{\ast}_{a, -s/n}f\right\|_{L^p}^p.
\end{split}
\end{equation}
Again, using Lemma \ref{sfga} with respect to the variable $t$, it follows that
\begin{align*}
& \big\||B(\cdot,1)|^{-s/n} \varphi_0(\sqrt{\mathscr{L}})f(\cdot)\big\|_{L^{p}(M)}^p  + \left\| \left\{  \big\||B(\cdot,2^{-j})|^{-s/n} \varphi(2^{-j}\sqrt{\mathscr{L}})f(\cdot)\big\|_{L^{p}(M)}^p \right\}_{j=1}^\infty \right\|_{\ell^{q/p}} \\
& \lesssim \left\|\big[\omega_{0}(\sqrt{\mathscr{L}})\big]^{\ast}_{a, -s/n}f\right\|_{L^p(M)}^p  +\left\|\left\{\left\| \left\|\big[\omega(2^{-j}t^{1/2}\sqrt{\mathscr{L}})\big]^{\ast}_{a, -s/n}f\right\|_{L^p(M)}^p \right\|_{L^{q/p}([1,4],dt/t)}  \right\}_{j=1}^{\infty} \right\|_{\ell^{q/p}}.
\end{align*}
From this we can also get \eqref{pmg} by direct computation.

\medskip

{\it Step 2.}
We show that if $a > \frac{n+n'}{\min\{p,q\}}$ then for $f \in \mathcal{S}_{\mathscr{L}}'$,
\begin{equation} \label{bess}
\begin{split}
& \left\|\big[\omega_{0}(\sqrt{\mathscr{L}})\big]^{\ast}_{a, -s/n}f \right\|_{L^{p}} +   \left( \int_{0}^{1}
\left\|\big[ \omega(t^{1/2}\sqrt{\mathscr{L}})\big]^{\ast}_{a, -s/n}f \right\|_{L^p}^{q}\frac{dt}{t}\right)^{1/q}\\
& \lesssim \big\||B(\cdot, 1)|^{-s/n}\omega_{0}(\sqrt{\mathscr{L}})f (\cdot)\big\|_{L^{p}} +   \left( \int_{0}^{1}
\big\| |B(\cdot, t^{1/2})|^{-s/n} \omega(t^{1/2}\sqrt{\mathscr{L}})f(\cdot)  \big\|_{L^p}^q \frac{dt}{t}\right)^{1/q}
\end{split}
\end{equation}
and
{\small \begin{equation} \label{lkjh}
\begin{split}
& \left\|\big[\omega_{0}(\sqrt{\mathscr{L}})\big]^{\ast}_{a, -s/n}f \right\|_{L^{p}} + \left\|  \left( \int_{0}^{1}
\left|\big[ \omega(t^{1/2}\sqrt{\mathscr{L}})\big]^{\ast}_{a, -s/n}f \right|^{q}\frac{dt}{t}\right)^{1/q}\right\|_{L^{p}}\\
& \lesssim \big\||B(\cdot, 1)|^{-s/n}\omega_{0}(\sqrt{\mathscr{L}})f (\cdot)\big\|_{L^{p}} + \left\|  \left( \int_{0}^{1}
 |B(\cdot, t^{1/2})|^{-sq/n} \big|\omega(t^{1/2}\sqrt{\mathscr{L}})f(\cdot) \big|^{q}\frac{dt}{t}\right)^{1/q}\right\|_{L^{p}}.
\end{split}
\end{equation}
It should be} mentioned that these inequalities do not follow from Lemma \ref{ptre}, since $\omega_{0}$ and $\omega$
do not have compact support.

We shall only give the details of the proof of \eqref{lkjh}; the proof of \eqref{bess} is analogous and will be omitted.

To prove \eqref{lkjh} we will use Lemma \ref{moui}.
Let $a >\frac{n+n'}{\min\{p,q\}}$. Let $r$ be a positive number satisfying $ar > n+ n'$ and $r <\min\{p,q\}$,
and let $N \in \mathbb{N}$ be sufficiently large such that
\begin{align} \label{qyy}
Nr - |s|rn'/n \geq ar \quad \mbox{and} \quad   2Nr-|s|r-n >0.
\end{align}

Let $\ell \geq 1$ and $t \in [1,4]$. Replacing $x$ by $y$ in \eqref{cen222}, then multiplying on both sides by
\begin{align*}
|B(y, 2^{-\ell}t^{1/2})|^{-sr/n}(1 + 2^{\ell} t^{-1/2}
\rho(x,y))^{-ar},
\end{align*}
and using the inequalities
\begin{align*}
|B(y, 2^{-\ell}t^{1/2})|^{-sr/n} &\lesssim |B(z, 2^{-\ell}t^{1/2})|^{-sr/n}(1 + 2^{\ell}t^{-1/2}\rho(y,z))^{|s|rn'/n}  \\
&\lesssim 2^{j|s|r}|B(z, 2^{-(j+\ell)}t^{1/2})|^{-sr/n}(1 + 2^{\ell}\rho(y,z))^{|s|rn'/n},
\end{align*}
$|B(z, 2^{-(j+ \ell)})|^{-1} \lesssim 2^{nj} |B(z,2^{-\ell})|^{-1}$ and \eqref{qyy}, we obtain
\begin{equation} \label{cen444}
\begin{split}
&\frac{|B(y, 2^{-\ell}t^{1/2})|^{-sr/n}\big|\omega(2^{-\ell}t^{1/2}f(y)\big|^{r}}{(1+2^{\ell}t^{-1}\rho(x,y))^{ar}} \\
&\lesssim \sum_{j=0}^{\infty} 2^{-(2Nr-|s|r-n) j}
\int_{M}  \frac{|B(z, 2^{-(j+\ell)}t^{1/2})|^{-sr/n}|\omega(2^{-(j+\ell)}t^{1/2}\sqrt{\mathscr{L}})f(z)|^{r}}
{{|B(z,2^{-\ell})|}(1+2^{\ell}\rho(y,z))^{ar}(1 +2^{\ell}t^{-1/2}\rho(x,y))^{ar}} d\mu(z).
\end{split}
\end{equation}
Taking the supremum over $y \in M$ on both sides, and using the fundamental inequality
\begin{align*}
(1+2^{\ell}\rho(y,z))^{ar}(1 + 2^{\ell}t^{-1/2}\rho(x,y))^{ar} \geq C (1 + 2^{\ell}\rho(x,z))^{ar}, \quad \forall \ell \geq 1, \ \forall t \in [1,4],
\end{align*}
we arrive at
\begin{equation} \label{cen4442}
\begin{split}
&\left|\big[ \omega(2^{-\ell}t^{1/2}\sqrt{\mathscr{L}})\big]^{\ast}_{a, -s/n}f(x)\right|^{r} \\
&\lesssim \sum_{j=0}^{\infty} 2^{-(2Nr-|s|r-n) j}
\int_{M}  \frac{|B(z, 2^{-(j+\ell)}t^{1/2})|^{-sr/n}|\omega(2^{-(j+\ell)}t^{1/2}\sqrt{\mathscr{L}})f(z)|^{r}}
{{|B(z,2^{-\ell})|}(1+2^{\ell}\rho(x,z))^{ar}} d\mu(z) \\
&= \sum_{j=\ell}^{\infty} 2^{-(2Nr-|s|r-n) (j -\ell)}
\int_{M}  \frac{|B(z, 2^{-j}t^{1/2})|^{-sr/n}|\omega(2^{-j}t^{1/2}\sqrt{\mathscr{L}})f(z)|^{r}}
{{|B(z,2^{-\ell})|}(1+2^{\ell}\rho(x,z))^{ar}} d\mu(z)  \\
&\leq \sum_{j=1}^{\infty} 2^{-(2Nr-|s|r-n) |j -\ell|}
\int_{M}  \frac{|B(z, 2^{-j}t^{1/2})|^{-sr/n}|\omega(2^{-j}t^{1/2}\sqrt{\mathscr{L}})f(z)|^{r}}
{{|B(z,2^{-\ell})|}(1+2^{\ell}\rho(x,z))^{ar}} d\mu(z) .
\end{split}
\end{equation}

Analogously, we can deduce from \eqref{cen333} in Lemma \ref{moui} the following estimate for the inhomogeneous term:
\begin{equation} \label{cen555}
\begin{split}
&\left|\big[ \omega_{0}(\sqrt{\mathscr{L}})\big]^{\ast}_{a, -s/n}f(x)\right|^{r} \\
 &\lesssim   \int_{M} \frac{|B(z, 1)|^{-sr/n}|\omega_{0}(\sqrt{\mathscr{L}})f(z)|^{r}}{|B(z,1)|(1+\rho(x,z))^{ar}} d\mu(z)  \\
& \quad  + \sum_{j=1}^{\infty} 2^{-(2Nr -|s|r-n) j }\int_{M}
\frac{|B(z,2^{-j}t^{1/2})|^{-sr/n} |\omega(2^{-j}t^{1/2}\sqrt{\mathscr{L}})f(z)|^{r}}{|B(z,1)|(1+\rho(x,z))^{ar}} d\mu(z) .
\end{split}
\end{equation}

Taking the norm $(\int_{1}^{4} |\cdot|^{q/r}\frac{dt}{t})^{r/q}$ on both sides of \eqref{cen555} (noting that the left-hand
side is independent of $t$), and using Minkowski's
inequality and Lemma \ref{maxxx}, we obtain
\begin{align*}
&\left|\big[ \omega_{0}(\sqrt{\mathscr{L}})\big]^{\ast}_{a, -s/n}f(x)\right|^{r} \\
&\lesssim \int_{M} \frac{|B(z,1)|^{-sr/n}|\omega_{0}(\sqrt{\mathscr{L}})f(z)|^{r}}{|B(z,1)|(1+\rho(x,z))^{ar}} d\mu(z)\\
& \quad  + \sum_{j=1}^{\infty} 2^{-(2Nr-|s|r-n) j} \int_{M}
\frac{\big(\int_{1}^{4}|B(z,2^{-j}t^{1/2})|^{-sq/n}|\omega(2^{-j}t^{1/2}\sqrt{\mathscr{L}})f(z)|^{q}
\frac{dt}{t}\big)^{r/q}}{|B(z,1)|(1+\rho(x,z))^{ar}} d\mu(z)\\
& \lesssim \mathcal{M} \left[ |B(\cdot, 1)|^{-sr/n} |\omega_0(\sqrt{\mathscr{L}})f(\cdot)|^r \right](x) \\
& \quad + \sum_{j=1}^{\infty} 2^{-(2Nr -|s|r -n)j} \mathcal{M} \left[  \left( \int_1^4 |B(\cdot,
2^{-j} t^{1/2}) |^{-sq/n} |\omega(2^{-j} t^{1/2} \sqrt{\mathscr{L}})f(\cdot)|^q \frac{dt}{t} \right)^{r/q} \right](x).
\end{align*}
Since $p/r >1$, $q /r >1$ and $2Nr-|s|r-n >0$, the last estimate along with Lemma \ref{osigg} and Lemma \ref{fsv} yields that
\begin{equation} \label{4967}
\begin{split}
&\left\|\big[ \omega_{0}(\sqrt{\mathscr{L}})\big]^{\ast}_{a, -s/n}f \right\|_{L^{p}} = \left\|\left|\big[\omega_{0}(\sqrt{\mathscr{L}})\big]^{\ast}_{a, -s/n}f \right|^{r}\right\|_{L^{p/r}}^{1/r} \\
&\lesssim \left\| \mathcal{M} \left[ |B(\cdot, 1)|^{-sr/n} |\omega_0(\sqrt{\mathscr{L}})f(\cdot)|^r \right]\right\|_{L^{p/r}}^{1/r} \\
&\quad +\left\| \left\| \left\{ \mathcal{M} \left[  \left( \int_1^4 |B(\cdot,
2^{-j} t^{1/2}) |^{-sq/n} |\omega(2^{-j} t^{1/2} \sqrt{\mathscr{L}})f(\cdot)|^q \frac{dt}{t} \right)^{r/q} \right]\right\}_{j=1}^{\infty}  \right\|_{\ell^{q/r}}\right\|_{L^{p/r}}^{1/r}  \\
&\lesssim \big\||B(\cdot, 1)|^{-sr/n} |\omega_0(\sqrt{\mathscr{L}})f(\cdot)|^r \big\|_{L^{p/r}}^{1/r}  \\
&\quad +\left\| \left\| \left\{ \left( \int_1^4 |B(\cdot,
2^{-j} t^{1/2}) |^{-sq/n} |\omega(2^{-j} t^{1/2} \sqrt{\mathscr{L}})f(\cdot)|^q \frac{dt}{t} \right)^{r/q} \right\}_{j=1}^{\infty}  \right\|_{\ell^{q/r}}\right\|_{L^{p/r}}^{1/r} \\
&=  \big\||B(\cdot, 1)|^{-s/n}\omega_{0}(\sqrt{\mathscr{L}})f(\cdot) \big\|_{L^{p}} \\
&\quad + \left\|  \left( \int_{0}^{1}
|B(\cdot, t^{1/2})|^{-sq/n} |\omega(t^{1/2}\sqrt{\mathscr{L}})f(\cdot)|^{q}\frac{dt}{t}\right)^{1/q}\right\|_{L^{p}}.
\end{split}
\end{equation}

Similarly (using Lemma \ref{sfga} instead of Lemma \ref{osigg}) we deduce from \eqref{cen4442} that
\begin{equation} \label{4968}
\begin{split}
 &\left\|  \left( \int_{0}^{1}\left|\big[ \omega(t^{1/2}\sqrt{\mathscr{L}})\big]^{\ast}_{a, -s/n}f \right|^{q}\frac{dt}{t}\right)^{1/q}\right\|_{L^{p}} \\
 &\quad\quad\quad\quad\quad\quad\quad\quad\quad  \lesssim  \left\|  \left( \int_{0}^{1}
|B(\cdot, t^{1/2})|^{-sq/n} |\omega(t^{1/2}\sqrt{\mathscr{L}})f(\cdot) |^{q}\frac{dt}{t}\right)^{1/q}\right\|_{L^{p}}.
\end{split}
\end{equation}

Combing \eqref{4967} and \eqref{4968} we obtain \eqref{lkjh}.

\medskip

\textit{Step 3.} We show that for any $a >0$,
 \begin{equation} \label{stepw}
\begin{split}
& \left\||B(\cdot, 1)|^{-s/n}\omega_{0}(\sqrt{\mathscr{L}})f(\cdot) \right\|_{L^{p}} +  \left( \int_{0}^{1}
\left\| |B(\cdot, t^{1/2})|^{-s/n} \omega(t^{1/2}\sqrt{\mathscr{L}})f (\cdot)\right\|_{L^p}^q \frac{dt}{t}\right)^{1/q} \\
& \lesssim  \left\|\big[\varphi_{0}(\sqrt{\mathscr{L}}) \big]^{\ast}_{a, -s/n}f\right\|_{L^{p}}  +
\left(\sum_{j =1}^{\infty} \left\| \big[\varphi(2^{-j}\sqrt{\mathscr{L}})\big]^{\ast}_{a, -s/n}f\right\|_{L^p}^{q}\right)^{1/q}
\quad  (\forall f \in \mathcal{S}_{\mathscr{L}}')
\end{split}
\end{equation}
and
{\small \begin{equation} \label{step}
\begin{split}
& \left\||B(\cdot, 1)|^{-s/n}\omega_{0}(\sqrt{\mathscr{L}})f(\cdot) \right\|_{L^{p}} + \left\|  \left( \int_{0}^{1}
 |B(\cdot, t^{1/2})|^{-sq/n}\big| \omega(t^{1/2}\sqrt{\mathscr{L}})f (\cdot)\big|^{q}\frac{dt}{t}\right)^{1/q}\right\|_{L^{p}} \\
& \lesssim  \left\|\big[\varphi_{0}(\sqrt{\mathscr{L}}) \big]^{\ast}_{a, -s/n}f\right\|_{L^{p}}  +
 \left\|\left(\sum_{j =1}^{\infty} \left| \big[\varphi(2^{-j}\sqrt{\mathscr{L}})\big]^{\ast}_{a, -s/n}f\right|^{q}\right)^{1/q}\right\|_{L^{p}}
 \quad  (\forall f \in \mathcal{S}_{\mathscr{L}}') .
\end{split}
\end{equation}}

We shall only give the details of the proof of \eqref{step}; the proof of \eqref{stepw} is analogous and will be omitted.

Applying the norm $\big(\int_1^4 |\cdot|^q\frac{dt}{t}\big)^{1/q}$
on both sides of \eqref{hrq} in Lemma \ref{Lemma3} (and noting that the right-hand side is independent of $t$) gives
\begin{equation} \label{yex}
\begin{split}
&\left(\int_{1}^{4}|B(x,2^{-\ell}t^{1/2})|^{-sq/n}|\omega(2^{-\ell}t^{1/2}\sqrt{\mathscr{L}})f(x) |^{q}  \frac{dt}{t}\right)^{1/q}\\
&\quad \lesssim 2^{-\ell \delta}\big[ \varphi_{0}(\sqrt{\mathscr{L}})\big]^{\ast}_{a,-s/n}f(x)
 + \sum_{j=1}^{\infty} 2^{-|j-\ell|\delta }\big[ \varphi(2^{-j}\sqrt{\mathscr{L}})\big]^{\ast}_{a,-s/n}f(x).
\end{split}
\end{equation}
We now put
\begin{align*}
\displaystyle G_{\ell } (x):= \begin{cases}
|B(x, 1)|^{-s/n}| \omega_{0}(\sqrt{\mathscr{L}})f(x)|, & \ell =0, \\
 \displaystyle \left(\int_{1}^{4}|B(x,2^{-\ell}t^{1/2})|^{-sq/n}|\omega(2^{-\ell}t^{1/2}\sqrt{\mathscr{L}})f(x) |^{q}  \frac{dt}{t}\right)^{1/q}, & \ell =1,2, \cdots
\end{cases}
\end{align*}
and
\begin{align*}
g_{j } (x):= \begin{cases}
\big[ \varphi_{0}(\sqrt{\mathscr{L}})\big]^{\ast}_{a,-s/n}f(x), & j =0, \\
\big[ \varphi(2^{-j}\sqrt{\mathscr{L}})\big]^{\ast}_{a,-s/n}f(x), & j =1,2, \cdots.
\end{cases}
\end{align*}
Then \eqref{inhomo} in Lemma \ref{Lemma3} coupled with \eqref{yex} implies
\begin{equation*}
G_{\ell} (x) \lesssim \sum_{j=0}^{\infty} 2^{-|j -\ell|\delta} g_{j}(x), \quad \ell =0,1,2,\cdots.
\end{equation*}
Applying Lemma \ref{sfga} in $L^{p}(\ell^{q})$ then yields \eqref{step}.

Now combining \eqref{pmg1}, \eqref{bess}, \eqref{stepw} and using Lemma \ref{ptre} (i) yields \eqref{hqy1},
while combining \eqref{pmg}, \eqref{lkjh}, \eqref{step} and using Lemma \ref{ptre} (ii) yields \eqref{hqy2}.
The proof of Theorem \ref{heat} is thus complete.
\end{proof}

\section{Proof of Theorem \ref{area}} \label{sec5}

We shall only give the proof of the quasi-norm equivalence \eqref{lusin2} for ``nonclassical'' spaces; the proof of \eqref{lusin1}
for ``classical'' spaces is easier and will be omitted.

The key step to prove \eqref{lusin2} is the following estimate.
\begin{lemma} \label{0962}
Let $s \in \mathbb{R}$, $0<p< \infty$, $0<q\leq \infty$ and $a > \frac{2n+2n'+1}{\min\{p,q\}}$. Then for $f \in \mathcal{S}_{\mathscr{L}}'$,
\begin{equation} \label{lu}
\begin{split}
& \left\|\left(\int_{0}^{1}\left|\big[ \omega(t^{1/2}\sqrt{\mathscr{L}})\big]^{\ast}_{a, -s/n}f\right|^{q} \frac{dt}{t} \right)^{1/q} \right\|_{L^{p}} \\
& \lesssim \left\| \left(\iint_{\Gamma^{loc}(\cdot)}  |B(z,t^{1/2})|^{-sq/n}|\omega(t^{1/2}\sqrt{\mathscr{L}})f(z)|^{q}
 \frac{d\mu(z)}{{|B(\cdot, t^{1/2})|}} \frac{dt}{t}\right)^{1/q} \right\|_{L^{p}},
 \end{split}
\end{equation}
where $\omega$ is defined by \eqref{uytr} and $\Gamma^{loc}(x):=\{(y,t) \in M \times (0,1] :  \rho(x,y) <t^{1/2}\}$.
\end{lemma}
\begin{proof}
Let $r$ be a positive number such that $r < \min\{p,q\}$ and $ar >2n + 2n' +1$.
Let $N \in \mathbb{N}$ such that $Nr - |s|rn'/n \geq ar$ and $2Nr-|s|r-2n>0$.

Note that for any integrable function $g$ on $M$ and any $u >0$, by Fubini's theorem we have
\begin{align*}
\int_{M}g(z)d\mu(z)&=\int_{M}g(z)\left(|B(z,u)|^{-1} \int_{M} \chi_{B(z,u)}(y)d\mu(y)\right)d\mu(z) \\
& = \int_{M}\left( \int_{M} g(z)|B(z,u)|^{-1} \chi_{B(y,u)}(z)d\mu(z)\right)d\mu(y).
\end{align*}
Using this identity, we can write \eqref{cen444} as follows:  for all $\ell \geq 1$ and $t \in [1,4]$,
\begin{equation} \label{hushu}
\begin{split}
&\left|\big[ \omega(2^{-\ell}t^{1/2}\sqrt{\mathscr{L}})\big]^{\ast}_{a, -s/n}f(x)\right|^{r} \\
&\lesssim \sum_{j=0}^{\infty} 2^{-(2Nr-|s|r -n) j}
\int_{M}\left(\int_{M}   \frac{|B(z, 2^{-(j+\ell)}t^{1/2})|^{-sr/n}|\omega(2^{-(j+\ell)}t^{1/2}\sqrt{\mathscr{L}})f(z)|^{r}}
{{|B(z,2^{-\ell})|}(1+2^{\ell}\rho(x,z))^{ar}} \right.  \\
&\ \ \ \ \ \ \ \ \ \ \ \ \ \ \ \ \ \ \ \ \ \ \ \ \ \ \ \ \ \ \ \ \ \ \ \ \ \ \ \ \ \ \ \ \ \ \ \ \ \ \ \ \ \ \ \
\left. \times \frac{\chi_{B(y,2^{-(j+\ell)}t^{1/2})}(z)}
{|B(z,2^{-(j+\ell)}t^{1/2})| } d\mu(z)\right) d\mu(y).
\end{split}
\end{equation}
Set
\begin{equation*}
H_{1}(x,y,z) := \frac{|B(z, 2^{-(j+\ell)}t^{1/2})|^{-sr/n}|\omega(2^{-(j+\ell)}t^{1/2}\sqrt{\mathscr{L}})f(z)|^{r}\chi_{B(y,2^{-(j+\ell)}t^{1/2})}(z)}
{|B(z,2^{-\ell})|^{r/q}(1+2^{\ell}\rho(x,z))^{ar- (n+n'+1)}|B(z,2^{-(j+\ell)}t^{1/2})|}
\end{equation*}
and
\begin{equation*}
H_{2}(x,z) := \frac{1}
{|B(z,2^{-\ell})|^{(q-r)/q}(1+2^{\ell}\rho(x,z))^{n+n'+1}}
\end{equation*}
Then from \eqref{hushu}, H\"{o}lder's inequality and \eqref{ges}, it follows that
{\small \begin{align*}
&\left|\big[ \omega(2^{-\ell}t^{1/2}\sqrt{\mathscr{L}})\big]^{\ast}_{a, -s/n}f(x)\right|^{r} \\
&\lesssim \sum_{j=0}^{\infty} 2^{-(2Nr-|s|r -n) j} \int_{M} \left( \int_{M}H_{1} (x,y,z)H_{2}(x,z) d\mu(z)\right)d\mu(y) \\
&\lesssim \sum_{j=0}^{\infty} 2^{-(2Nr-|s|r -n) j} \int_{M} \left( \int_{M} |H_{1} (x,y,z)|^{\frac{q}{r}}d\mu(z)\right)^{\frac{r}{q}}
\left(\int_{M}|H_{2}(x,z)|^{\frac{q}{q-r}} d\mu(z)\right)^{\frac{q-r}{q}}d\mu(y) \\
&\lesssim \sum_{j=0}^{\infty} 2^{-(2Nr -|s|r-n)j}
\int_{M}\left(\int_{M}   \frac{|B(z, 2^{-(j+\ell)}t^{1/2})|^{-sq/n}|\omega(2^{-(j+\ell)}t^{1/2}\sqrt{\mathscr{L}})f(z)|^{q}}
{{|B(z,2^{-\ell})|}(1+2^{\ell}\rho(x,z))^{aq-(n+n'+1)q/r} }  \right.\\
& \ \ \ \ \ \ \ \ \ \ \ \ \ \ \ \ \  \ \ \ \ \ \ \ \ \ \ \ \ \ \ \ \ \
 \left. \times \frac{\chi_{B(y,2^{-(j+\ell)}t^{1/2})}(z) }{ |B(z,2^{-(j+\ell)}t^{1/2})|^{q/r} }
 d\mu(z)\right)^{r/q} d\mu(y), \quad \forall \ell \geq 1, \ t \in [1,4].
\end{align*}}
Taking the norm $\big(\int_{1}^{4} |\cdot|^{q/r}\frac{dt}{t}\big)^{r/q}$ on both sides,
and using Minkowski's inequality,  we get
{\small \begin{align} \label{0852}
&  \left(\int_{1}^{4}\left|\big[ \omega(2^{-\ell}t^{1/2}\sqrt{\mathscr{L}})\big]^{\ast}_{a, -s/n}f(x)\right|^{q}
\frac{dt}{t} \right)^{r/q} \nonumber\\
&\lesssim \sum_{j=0}^{\infty} 2^{-(2Nr-|s|r-n) j}
\int_{M}\left[\left(\int_{1}^{4}\int_{M}   \frac{|B(z, 2^{-(j+\ell)}t^{1/2})|^{-sq/n}|\omega(2^{-(j+\ell)}t^{1/2}
\sqrt{\mathscr{L}})f(z)|^{q}}
{{|B(z,2^{-\ell})|}(1+2^{\ell}\rho(x,z))^{aq- (n+n'+1)q/r}} \right.\right.   \nonumber \\
& \ \ \ \ \ \ \ \ \ \ \ \ \ \ \ \ \ \ \ \ \ \ \ \ \ \ \ \ \ \ \ \  \ \
\left.\left.\times \frac{  \chi_{B(y,2^{-(j+\ell)}t^{1/2})}(z)  }{  |B(z,2^{-(j+\ell)}t^{1/2})|^{q/r}  }
d\mu(z) \frac{dt}{t}\right)^{r/q}\right] d\mu(y)\\
&= \sum_{j=0}^{\infty} 2^{-(2Nr-|s|r-n) j} \int_{M}\left[\left(\int_{1}^{4}\int_{\rho(y,z)<2^{-(j+\ell)}t^{1/2}}\frac{ |B(z, 2^{-(j+\ell)}t^{1/2})|^{-sq/n}  }{ |B(z,2^{-\ell})|(1+2^{\ell}\rho(x,z))^{aq- (n+n'+1)q/r}  } \right.\right.  \nonumber\\
 & \quad\quad\quad\quad\quad\quad\quad\quad\quad\quad\quad\quad \times
 \left.\left.\frac{|\omega(2^{-(j+\ell)}t^{1/2}\sqrt{\mathscr{L}})f(z)|^{q}}
{|B(z,2^{-(j+\ell)}t^{1/2})|^{(q/r)-1} }\frac{d\mu(z)}{|B(z,2^{-(j+\ell)}t^{1/2})|  } \frac{dt}{t}\right)^{r/q}\right] d\mu(y). \nonumber
 \end{align}
Note} that if $\rho(y,z)<2^{-(j+\ell)}t^{1/2}$ then
\begin{equation*}
\frac{1}{(1 + 2^{\ell}(x,z))^{aq -(n+n'+1)q/r}} \lesssim \frac{(1 + 2^{\ell}\rho(y,z))^{aq -(n+n'+1)q/r}}
{(1 + 2^{\ell}\rho(x,y))^{aq -(n+n'+1)q/r}} \lesssim
\frac{1}{(1 + 2^{\ell}\rho(x,y))^{aq -(n+n'+1)q/r}}.
\end{equation*}
Also note that for all $\ell \geq 1$ and $t \in [1,4]$,
\begin{align*}
\frac{1}{|B(z,2^{-\ell})||B(z,2^{-(j+\ell)}t^{1/2})|^{(q/r)-1}}&\lesssim \frac{2^{[(q/r)-1]nj}}{|B(z,2^{-\ell})|^{q/r}}\\
& \lesssim
\frac{2^{[(q/r)-1]nj}(1 + 2^{\ell}\rho(y,z))^{n'q/r}}{|B(y,2^{-\ell})|^{q/r}} \lesssim \frac{2^{(q/r)nj}}{|B(y,2^{-\ell})|^{q/r}}.
\end{align*}
Inserting these estimates into \eqref{0852},  we get
{\small \begin{align*}
&\left(\int_{1}^{4}\left|\big[ \omega(2^{-\ell}t^{1/2}\sqrt{\mathscr{L}})\big]^{\ast}_{a, -s/n}f(x)\right|^{q}
\frac{dt}{t} \right)^{r/q} \\
&\lesssim  \sum_{j=0}^{\infty} 2^{-(2Nr-|s|r-n) j}2^{nj} \int_{M}
\left[\left(\int_{1}^{4}\int_{\rho(y,z)<2^{-(j+\ell)}t^{1/2}} |B(z, 2^{-(j+\ell)}t^{1/2})|^{-sq/n}|\omega(2^{-(j+\ell)}t^{1/2}\sqrt{\mathscr{L}})f(z)|^{q}\right.\right.    \\
& \quad\quad\quad\quad\quad\quad\quad\quad\quad\quad\quad\quad
 \times \left.\left.
\frac{d\mu(z)}{|B(z,2^{-(j+\ell)}t^{1/2})|  } \frac{dt}{t} \right)^{r/q} \right]   \frac{1}{|B(y,2^{-\ell})|
(1 +2^{\ell}\rho(x,y))^{ar -(n+n'+1)} }d\mu(y).
\end{align*}
Since} $ar - (n + n'+ 1) > n +n'$, we further apply Lemma \ref{maxxx} to conclude that
for all $\ell \geq 1$,
{\small \begin{align*}
& \left(\int_{1}^{4}\left|\big[ \omega(2^{-\ell}t^{1/2}\sqrt{\mathscr{L}})\big]^{\ast}_{a, -s/n}f(x)\right|^{q}
\frac{dt}{t} \right)^{r/q} \\
&\lesssim \sum_{j=0}^{\infty} 2^{-(2Nr-|s|r-2n) j}
\mathcal{M} \left[\left(\int_{1}^{4}\int_{\rho(\cdot,z)<2^{-(j+\ell)}t^{1/2}}   |B(z, 2^{-(j+\ell)}t^{1/2})|^{-sq/n}|
\omega(2^{-(j+\ell)}t^{1/2}\sqrt{\mathscr{L}})f(z)|^{q} \right.\right.\\
&  \quad\quad\quad\quad\quad\quad\quad\quad\quad\quad\quad\quad\quad\quad\quad   \times
\left.\left.\frac{d\mu(z)}{|B(z,2^{-(j+\ell)}t^{1/2})|  } \frac{dt}{t} \right)^{r/q}  \right](x)\\
&\leq \sum_{j=1}^{\infty} 2^{-|j -\ell|\delta}
\mathcal{M} \left[\left(\int_{1}^{4}\int_{\rho(\cdot,z)<2^{-j}t^{1/2}}   |B(z, 2^{-j}t^{1/2})|^{-sq/n}|
\omega(2^{-j}t^{1/2}\sqrt{\mathscr{L}})f(z)|^{q} \frac{d\mu(z)}{|B(z,2^{-j}t^{1/2})|  } \frac{dt}{t} \right)^{r/q}  \right](x),
\end{align*}
where} $\delta:= 2Nr-|s|r-2n>0$.
Then, since $p/r >1$, $q/r >1$ and $2Nr-|s|r-2n>0$, we apply
Lemma \ref{sfga} in $L^{p/r}(\ell^{q/r})$ and Lemma \ref{fsv}, to get
{\small \begin{align*}
&\left\|\left(\int_{0}^{1}\left|\big[ \omega(t^{1/2}\sqrt{\mathscr{L}})\big]^{\ast}_{a, -s/n}f\right|^{q}
 \frac{dt}{t} \right)^{1/q} \right\|_{L^{p}} \\
&= \left\|\left\{ \left(\int_{1}^{4}\left|\big[ \omega(2^{-\ell}t^{1/2}\sqrt{\mathscr{L}})
\big]^{\ast}_{a, -s/n}f(x)\right|^{q} \frac{dt}{t} \right)^{r/q} \right\}_{\ell =1}^{\infty}
 \right\|_{L^{p/r}(\ell^{q/r})}^{1/r} \\
& \lesssim \left\|\left\{ \mathcal{M} \left[\left(\int_{1}^{4}\int_{\rho(\cdot,z)<2^{-j}t^{1/2}}   |B(z, 2^{-j}t^{1/2})|^{-sq/n}|
\omega(2^{-j}t^{1/2}\sqrt{\mathscr{L}})f(z)|^{q} \frac{d\mu(z)}{|B(z,2^{-j}t^{1/2})|  } \frac{dt}{t} \right)^{r/q}
\right] \right\}_{j =1}^{\infty}  \right\|_{L^{p/r}(\ell^{q/r})}^{1/r} \\
& \lesssim \left\|\left\{ \left(\int_{1}^{4}\int_{\rho(\cdot,z)<2^{-j}t^{1/2}}   |B(z, 2^{-j}t^{1/2})|^{-sq/n} |
\omega(2^{-j}t^{1/2}\sqrt{\mathscr{L}})f(z)|^{q} \frac{d\mu(z)}{|B(z,2^{-j}t^{1/2})|  }
\frac{dt}{t} \right)^{r/q} \right\}_{j =1}^{\infty}  \right\|_{L^{p/r}(\ell^{q/r})}^{1/r} \\
& = \left\|\left(\sum_{j=1}^{\infty}  \int_{2^{-2j}}^{2^{-2(j-1)}}\int_{\rho(\cdot,z)<t^{1/2}}
   |B(z, t^{1/2})|^{-sq/n} |
\omega(t^{1/2}\sqrt{\mathscr{L}})f(z)|^{q} \frac{d\mu(z)}{|B(z,t^{1/2})|  }
\frac{dt}{t}   \right)^{r/q}  \right\|_{L^{p/r}}^{1/r} \\
& =\left\| \left(\iint_{\Gamma^{loc}(\cdot)}  |B(z,t^{1/2})|^{-sq/n}|\omega(t^{1/2}\sqrt{\mathscr{L}})f(z)|^{q}
 \frac{d\mu(z)}{{|B(z, t^{1/2})|}} \frac{dt}{t}\right)^{1/q} \right\|_{L^{p}}.
\end{align*}
This} implies \eqref{lu} since $|B(z, t^{1/2})| \sim |B(x, t^{1/2})|$ when $\rho(x,z) < t^{1/2}$.
The proof of Lemma \ref{0962} is complete.
\end{proof}

The next estimate is a converse of the one stated in the previous lemma.
\begin{lemma} \label{0963}
Let $s \in \mathbb{R}$, $0<p< \infty$, $0<q\leq \infty$ and $a > 0$.  Then for $f \in \mathcal{S}_{\mathscr{L}}'$, we have
\begin{equation} \label{lu2}
\begin{split}
& \left\| \left(\iint_{\Gamma^{loc}(\cdot)}  |B(z,t^{1/2})|^{-sq/n}|\omega(t^{1/2}\sqrt{\mathscr{L}})f(z)|^{q}
 \frac{d\mu(z)}{{|B(\cdot, t^{1/2})|}} \frac{dt}{t}\right)^{1/q} \right\|_{L^{p}}\\
& \quad \lesssim \left\|\left(\int_{0}^{1}\left|\big[ \omega(t^{1/2}\sqrt{\mathscr{L}})\big]^{\ast}_{a, -s/n}f\right|^{q}
\frac{dt}{t} \right)^{1/q} \right\|_{L^{p}}.
 \end{split}
\end{equation}
\end{lemma}
\begin{proof}
Observe that for all $a >0$, $t \in (0,1]$ and $x \in M$,
\begin{align*}
&\frac{1}{|B(x, t^{1/2})|} \int_{B(x,t^{1/2})}    |B(z,t^{1/2})|^{-sq/n}|\omega(t^{1/2}\sqrt{\mathscr{L}})f(z)|^{q}d\mu(z)  \\
&\quad\quad \leq \sup_{z \in B(x, t^{1/2})}
|B(z,t^{1/2})|^{-sq/n}|\omega(t^{1/2}\sqrt{\mathscr{L}})f(z)|^{q}
\leq 2^{aq} \left|\big[\omega(t^{1/2}\sqrt{\mathscr{L}}) \big]^{\ast}_{a, -s/n}f(x)\right|^{q}.
\end{align*}
Taking the norm $\int_{0}^{1}|\cdot| \frac{dt}{t}$ on both sides gives the pointwise estimate
\begin{align*}
\iint_{\Gamma^{loc}(x)} |B(z,t^{1/2})|^{-sq/n}|\omega(t^{1/2}\sqrt{\mathscr{L}})f(z)|^{q}
 \frac{d\mu(z)}{{|B(x, t^{1/2})|}} \frac{dt}{t}
 \leq 2^{aq} \int_{0}^{1}  \left|\big[\omega(t^{1/2}\sqrt{\mathscr{L}}) \big]^{\ast}_{a}f(x)\right|^{q} \frac{dt}{t},
\end{align*}
which readily yields the estimate \eqref{lu2}.
\end{proof}

Now we are in a position to complete the proof of Theorem \ref{area}.
\begin{proof}[Proof of Theorem \ref{area}]
Let $a > \frac{2n+2n'+1}{\min\{p,q\}}$. Then, using Theorem \ref{heat},  the obvious estimate
{\small $|B(x, t^{1/2})|^{-s/n} \big|\omega(t^{1/2}\sqrt{\mathscr{L}})f(x) \big|
\lesssim \big[ \omega(t^{1/2}\sqrt{\mathscr{L}})\big]^{\ast}_{a, -s/n}f (x)$},
Lemma \ref{0962}, Lemma \ref{0963} and \eqref{lkjh}, we have
{\small \begin{align*}
\|f\|_{\widetilde{F}^{s}_{p,q}(\mathscr{L})} &\sim  \big\||B(\cdot, 1)|^{-s/n}\omega_{0}(\sqrt{\mathscr{L}})f
(\cdot)\big\|_{L^{p}} + \left\|  \left( \int_{0}^{1}
|B(\cdot, t^{1/2})|^{-sq/n} \big|\omega(t^{1/2}\sqrt{\mathscr{L}})f(\cdot) \big|^{q}\frac{dt}{t}\right)^{1/q}\right\|_{L^{p}}\\
& \leq \big\||B(\cdot, 1)|^{-s/n}\omega_{0}(\sqrt{\mathscr{L}})f (\cdot)\big\|_{L^{p}} +\left\|  \left( \int_{0}^{1}
\left|\big[ \omega(t^{1/2}\sqrt{\mathscr{L}})\big]^{\ast}_{a, -s/n}f \right|^{q}\frac{dt}{t}\right)^{1/q}\right\|_{L^{p}} \\
& \lesssim \big\||B(\cdot, 1)|^{-s/n}\omega_{0}(\sqrt{\mathscr{L}})f (\cdot)\big\|_{L^{p}}\\
& \quad+\left\| \left(\iint_{\Gamma^{loc}(\cdot)}  |B(z,t^{1/2})|^{-sq/n}\big|\omega(t^{1/2}\sqrt{\mathscr{L}})f(z)\big|^{q}
 \frac{d\mu(z)}{{|B(\cdot, t^{1/2})|}} \frac{dt}{t}\right)^{1/q} \right\|_{L^{p}}\\
 & \lesssim \big\||B(\cdot, 1)|^{-s/n}\omega_{0}(\sqrt{\mathscr{L}})f(\cdot) \big\|_{L^{p}} +\left\|  \left( \int_{0}^{1}
\left|\big[ \omega(t^{1/2}\sqrt{\mathscr{L}})\big]^{\ast}_{a_{1}, -s/n}f \right|^{q}\frac{dt}{t}\right)^{1/q}\right\|_{L^{p}} \\
& \lesssim \big\||B(\cdot, 1)|^{-s/n}\omega_{0}(\sqrt{\mathscr{L}})f (\cdot)\big\|_{L^{p}}+ \left\|  \left( \int_{0}^{1}
|B(\cdot, t^{1/2})|^{-sq/n} \big|\omega(t^{1/2}\sqrt{\mathscr{L}})f(\cdot) \big|^{q}\frac{dt}{t}\right)^{1/q}\right\|_{L^{p}} \\
& \sim \|f\|_{\widetilde{F}^{s}_{p,q}(\mathscr{L})},
\end{align*}}
which yields \eqref{lusin2}. The proof of Theorem \ref{area} is thus complete.
\end{proof}

\medskip
\section*{Acknowledgments}
We sincerely thank the anonymous
referees for the careful reading of the manuscript and
several constructive suggestions which improved the exposition of this paper.
We also thank Tino Ullrich and Martin Sch\"{a}fer for inspiring discussions
on the estimate for the inhomogeneous term in Lemma \ref{moui}.





\end{document}